\RequirePackage[l2tabu, orthodox]{nag}

\documentclass[
11pt,                          
english                        
]{article}

\let\hbarorig\hbar

\usepackage[english]{babel}    
\usepackage{amsmath}           
\usepackage[utf8]{inputenc}    
\usepackage[T1]{fontenc}       
\usepackage{longtable}         
\usepackage{exscale}           
\usepackage[final]{graphicx}   
\usepackage[sort]{cite}        
\usepackage{eucal}
\usepackage{array}             
\usepackage[a4paper]{geometry} 
\usepackage[multiuser]{fixme}  
\usepackage{xspace}            
\usepackage{tikz}              
\usepackage{amssymb}           
\usepackage{bbm}               
\usepackage{mathtools}         
\usepackage[amsmath,thmmarks,hyperref]{ntheorem} 
\usepackage{mathrsfs}          
\usepackage{stmaryrd}          
\usepackage{paralist}          
\usepackage{aliascnt}          
\usepackage{enumitem}
\usepackage[expansion=false    
           ]{microtype}        
\usepackage[nottoc]{tocbibind} 
\usepackage[
           final=true,         
           pdfpagelabels       
           ]{hyperref}         

\usetikzlibrary{matrix}
\usetikzlibrary{arrows}
\usetikzlibrary{patterns}
\usetikzlibrary{decorations.pathreplacing}

\author{
  \textbf{Chiara Esposito}\thanks{\texttt{chesposito@unisa.it}},\\[0.3cm]
   Dipartimento di Matematica\\
   Università degli Studi di Salerno\\
   via Giovanni Paolo II, 123\\
   84084 Fisciano (SA)\\
   Italy \\[0.5cm]
  \textbf{Philipp Schmitt}\thanks{\texttt{philippschmitt@math.ku.dk}},\\[0.3cm]
  Department of Mathematical Sciences\\
  University of Copenhagen\\
  Universitetsparken 5\\
  DK-2100 København \O{}\\ 
  Denmark\\[0.5cm]
  \textbf{Stefan Waldmann}\thanks{\texttt{stefan.waldmann@mathematik.uni-wuerzburg.de}},\\[0.3cm]
  Institut für Mathematik\\
  Lehrstuhl für Mathematik X\\
  Universität Würzburg\\
  Campus Hubland Nord\\
  Emil-Fischer-Straße 31\\
  97074 Würzburg\\
  Germany
}

\renewcommand{\mathbb}[1]{\mathbbm{#1}}

\newcommand{\refitem}[1] {\textit{\ref{#1}.)}}

\numberwithin{equation}{section}

\allowdisplaybreaks

\renewcommand{\arraystretch}{1.2}

\let\originalleft\left
\let\originalright\right
\renewcommand{\left}{\mathopen{}\mathclose\bgroup\originalleft}
\renewcommand{\right}{\aftergroup\egroup\originalright}

\renewcommand{\cleardoublepage}{\clearpage\ifodd\c@page\else\vspace*{\fill}\thispagestyle{empty}\newpage\fi}

\theoremheaderfont{\normalfont\bfseries}
\theorembodyfont{\itshape}
\newtheorem{lemma}{Lemma}[section]

\newaliascnt{proposition}{lemma}
\newtheorem{proposition}[proposition]{Proposition}
\aliascntresetthe{proposition}

\newaliascnt{thm}{lemma}
\newtheorem{theorem}[thm]{Theorem}
\aliascntresetthe{thm}

\newaliascnt{corollary}{lemma}
\newtheorem{corollary}[corollary]{Corollary}
\aliascntresetthe{corollary}

\newaliascnt{definition}{lemma}
\newtheorem{definition}[definition]{Definition}
\aliascntresetthe{definition}

\newaliascnt{claim}{lemma}

\aliascntresetthe{claim}

\theorembodyfont{\rmfamily}

\newaliascnt{example}{lemma}

\aliascntresetthe{example}

\newaliascnt{remark}{lemma}
\newtheorem{remark}[remark]{Remark}
\aliascntresetthe{remark}

\theorembodyfont{\itshape}
\theoremnumbering{Roman}
\newtheorem{maintheorem}{Main Theorem}

\makeatletter
\def\theorem@checkbold{}
\makeatother

\theoremheaderfont{\scshape}
\theorembodyfont{\normalfont}
\theoremstyle{nonumberplain}
\theoremseparator{:}
\theoremsymbol{\hbox{$\boxempty$}}
\newtheorem{proof}{Proof}
\theoremsymbol{\hbox{$\triangledown$}}

\newenvironment{lemmalist}{\begin{compactenum}[\itshape i.)]}{\end{compactenum}}
\newenvironment{theoremlist}{\begin{compactenum}[\itshape i.)]}{\end{compactenum}}
\newenvironment{propositionlist}{\begin{compactenum}[\itshape i.)]}{\end{compactenum}}
\newenvironment{definitionlist}{\begin{compactenum}[\itshape i.)]}{\end{compactenum}}

\newcommand{\I}              {\mathrm{i}}
\newcommand{\E}              {\mathrm{e}}
\newcommand{\D}              {\mathop{}\!\mathrm{d}}
\newcommand{\Unit}           {\mathbb{1}}
\newcommand{\group}[1]        {\mathrm{#1}}
\newcommand{\algebra}[1]      {\mathscr{#1}}
\newcommand{\lie}[1]          {\mathfrak{#1}}
\DeclareMathOperator{\Pol}    {\mathrm{Pol}}
\DeclareMathOperator{\ad}     {\mathrm{ad}}
\DeclareMathOperator{\Ad}     {\mathrm{Ad}}
\newcommand{\Fun}[1][k]      {\mathscr{C}^{#1}}
\newcommand{\Cinfty}         {\Fun[\infty]}
\newcommand{\acts}            {\mathbin{\triangleright}}
\newcommand{\Sec}[1][k]      {\Gamma^{#1}}
\newcommand{\Secinfty}       {\Sec[\infty]}
\DeclareMathOperator{\Diffop}         {\mathrm{DiffOp}}
\newcommand{\argument}       {\,\cdot\,}
\newcommand{\at}[1]          {\big|_{#1}}
\newcommand{\At}[1]          {\Big|_{#1}}
\newcommand{\pr}             {\mathrm{pr}}
\newcommand{\tensor}[1][{}]           {\mathbin{\otimes_{\scriptscriptstyle{#1}}}}
\newcommand{\Lie}   {\mathscr{L}}
\DeclarePairedDelimiter{\abs}{\lvert}{\rvert}
\DeclarePairedDelimiter{\norm}{\lVert}{\rVert}
\newcommand{\Sym}                     {\mathrm{S}}
\DeclareMathOperator{\diag}  {\mathrm{diag}}
\newcommand{\ev}             {\mathrm{ev}}

\theoremnumbering{arabic}
\theoremheaderfont{\scshape}
\theoremsymbol{\hbox{$\boxempty$}}
\theoremstyle{empty}
\newtheorem{proofof}{}

\mathchardef\mhyphen="2D

\newcommand{\anfa}{``}

\newcommand{\anfel}{''\ }

\newcommand{\glkomma}{\, ,}
\newcommand{\glpunkt}{\, .}
\newcommand{\textueber}[2]{\overset{\mathclap{\strut{#2}}}#1}

\newcommand{\toFunction}[1]{#1^\flat}
\newcommand{\toFunctionPower}[2]{(#1^\flat)^{#2}}

\newcommand{\formParam}{z}

\newcommand{\univ}{\algebra U}
\newcommand{\univdeg}[1]{\algebra U_{#1}}

\newcommand{\SLN}[1]{\group{SL}(#1,\mathbb C)}
\newcommand{\sln}[1]{\lie{sl}_{#1}(\mathbb C)}
\newcommand{\SUN}[1]{\group{SU}(#1)}
\newcommand{\sunc}[1]{\lie{su}_{#1}}
\newcommand{\UN}[1]{\group{U}(#1)}

\newcommand{\SLC}{\SLN 2}
\newcommand{\slc}{\sln 2}
\newcommand{\SU}{\SUN 2}
\newcommand{\su}{\sunc 2}
\newcommand{\U}{\UN 1}

\newcommand{\numberofbars}{\overline n}
\newcommand{\numberofnonbars}{n}
\newcommand{\numberofbarsext}{\hat{\numberofbars}}
\newcommand{\numberofnonbarsext}{\hat{\numberofnonbars}}

\newcommand{\polefamily}{P}
\newcommand{\poleset}{\Omega}

\newcommand{\prodP}{*_\hbar^{\algebra P}}
\newcommand{\prodR}{*_\hbar^{\algebra R}}

\newcommand{\orbsign}{\mathcal{O}}
\newcommand{\orb}[1]{\orbsign_{#1}}
\newcommand{\coorbsign}{\Omega}
\newcommand{\coorb}[1]{\coorbsign_{#1}}

\newcommand{\vanideal}{\algebra I_{\mathrm{van}}}

\newcommand{\Polcompleteds}[2]{\smash{\widehat{\Pol}}_{#1}(#2)}

\newcommand{\leftinv}[1]{{#1}^{\mathrm{left}}} 

\newcommand{\withhbar}[1]{#1^{(\hbar)}}
\newcommand{\opl}{\ell}
\newcommand{\oplh}{\withhbar{\opl}}
\newcommand{\oplhs}{\smash{\oplh}}
\newcommand{\opr}{r}
\newcommand{\oprh}{\withhbar{\opr}}

\newcommand {\Opl}{L}
\newcommand {\Oplh}{\withhbar{\Opl}}
\newcommand{\Oplhs}{\smash{\Oplh}}
\newcommand {\Opr}{R}
\newcommand {\Oprh}{\withhbar{\Opr}}

\newcommand{\Oplhlift}{\withhbar{\hat L}}
\newcommand{\Oplhslift}{\smash{\Oplhlift}}

\newcommand{\defspace}{\algebra A_\hbar}

\newcommand{\quotientTR}[1][R]{quotient-$T_{#1}$-topology}
\newcommand{\reductionTR}[1][R]{reduction-$T_{#1}$-topology}
\newcommand{\quotientTRs}[1][R]{{\quotientTR[#1]} }

\newcommand{\alekseevN}[1]{\lie n^{#1}}
\newcommand{\mynorm}{\lie p}
\newcommand{\cpn}[1]{\mathbb{CP}^{#1}}

\newcommand{\evaluate}[2]{\langle{#1},{#2} \rangle}
\newcommand{\evaluatebig}[2]{\left\langle{#1},{#2}\right\rangle}

\newcommand{\finitesgroup}[1][2]{\algebra P_{\SUN{#1}}}
\newcommand{\finitesspace}[1][2]{\algebra P_{\mathbb C^{#1}}}
\newcommand{\finiteinvsgroup}[1][2]{\algebra R_{\SUN{#1}}}
\newcommand{\finiteinvsspace}[1][2]{\algebra R_{\mathbb C^{#1}}}

\newcommand{\partt}{\frac{\D}{\D t} \At{t=0}}
\newcommand{\parttsmall}{\frac{\D}{\D t} \at{t=0}}
\newcommand{\parttj}[1]{\frac{\D}{\D t_{#1}} \At{t_{#1}=0}}

\addto\extrasenglish{}

\newlength{\hatchspread}
\newlength{\hatchthickness}
\newlength{\hatchshift}
\newcommand{\hatchcolor}{}
\tikzset{hatchspread/.code={\setlength{\hatchspread}{#1}},
	hatchthickness/.code={\setlength{\hatchthickness}{#1}},
	hatchshift/.code={\setlength{\hatchshift}{#1}},
	hatchcolor/.code={\renewcommand{\hatchcolor}{#1}}}
\tikzset{hatchspread=3pt,
	hatchthickness=0.4pt,
	hatchshift=0pt,
	hatchcolor=black}
\pgfdeclarepatternformonly[\hatchspread,\hatchthickness,\hatchshift,\hatchcolor]
{custom north west lines}
{\pgfqpoint{\dimexpr-2\hatchthickness}{\dimexpr-2\hatchthickness}}
{\pgfqpoint{\dimexpr\hatchspread+2\hatchthickness}{\dimexpr\hatchspread+2\hatchthickness}}%
{\pgfqpoint{\dimexpr\hatchspread}{\dimexpr\hatchspread}}
{
	\pgfsetlinewidth{\hatchthickness}
	\pgfpathmoveto{\pgfqpoint{0pt}{\dimexpr\hatchspread+\hatchshift}}
	\pgfpathlineto{\pgfqpoint{\dimexpr\hatchspread+0.15pt+\hatchshift}{-0.15pt}}
	\ifdim \hatchshift > 0pt
	\pgfpathmoveto{\pgfqpoint{0pt}{\hatchshift}}
	\pgfpathlineto{\pgfqpoint{\dimexpr0.15pt+\hatchshift}{-0.15pt}}
	\fi
	\pgfsetstrokecolor{\hatchcolor}
	\pgfusepath{stroke}
}

\let\hbar\hbarorig

\title{Comparison and Continuity of Wick-type Star Products on certain coadjoint orbits}

\date{ }

\begin{document}

%
%

\maketitle

%
%

\begin{abstract}
  In this paper we discuss continuity properties of the Wick-type star product 
  on the 2-sphere, interpreted as a coadjoint orbit. Star products on 
  coadjoint orbits in general have been constructed by different techniques.
  We compare the constructions of Alekseev-Lachowska and Karabegov
  and we prove that they agree in general. In the case of the 2-sphere we 
  establish the continuity of the star product, thereby allowing for a  
  completion to a Fr\'echet algebra.
\end{abstract}

\newpage

%
%

\tableofcontents

%
%

\section{Introduction}
\label{sec:Introduction}

There are two main directions in the field of deformation quantization: 
\emph{formal} and \emph{strict} quantization. The first of these was introduced 
in \cite{bayen.et.al:1978a} and has a rather algebraic flavour. One 
considers an associative product for formal power series of smooth functions on 
a Poisson manifold as a deformation of the commutative product of such 
functions, in the direction of the Poisson bracket. Existence and 
classification of these formal quantizations on 
symplectic manifolds have been understood  
\cite{bertelson.cahen.gutt:1997a, dewilde.lecomte:1983b, fedosov:1994a, 
nest.tsygan:1995a} and the general case of Poisson manifolds was solved by 
Kontsevich's formality theorem \cite{kontsevich:2003a}. However, to be 
physically relevant formal deformation quantization lacks one important 
property: one cannot just plug in a real value representing Planck's constant 
due to convergence issues of the formal power series.

Strict deformation quantization \cite{natsume.nest.peter:2003a, 
natsume.nest:1999a, rieffel:1993a} attempts to overcome this problem by looking 
at fields of algebras defined for a suitable set of values for the deformation 
parameter. However, the situation in this context is much more unclear than in 
the formal setting and there are several competing definitions on what strict 
quantization should be. A prominent one is due to Rieffel 
\cite{rieffel:1993a}, who works with continuous fields of C*-algebras. In his 
work, he constructs such a field out of an isometric action of 
$\mathbb R^d$ on an arbitrary C*-algebra. Although this setting is fairly 
general, important symplectic manifolds like the 2-sphere with its 
$\group{SO}(3)$-invariant symplectic structure cannot be treated in this 
approach \cite{rieffel:1998a}. The method introduced by Rieffel, also used by 
Bieliavsky and Gayral in \cite{bieliavsky.gayral} for the case of negatively 
curved K\"ahlerian Lie groups, relies on oscillatory integrals, which makes 
this approach rather technical and impossible to generalize to infinite 
dimensional settings (as needed when 
studying quantum field theories), see also the recent work 
\cite{bieliavsky:riemannsurfaces}.

Another approach tries to introduce locally convex topologies 
\cite{waldmann:2014a, beiser.roemer.waldmann:2007a, beiser.waldmann:2014a} in 
which the 
formal power series converge on sufficiently large subalgebras. Several 
examples have been studied in this area, including also infinite dimensional 
ones like star products of exponential type in \cite{schoetz.waldmann:2018} and 
the Gutt star product on 
the dual of a Lie algebra in \cite{esposito.stapor.waldmann:2017a}. The 
hyperbolic disc was treated in \cite{shoetz}. 
However, the work so far is based on the study of specific examples 
and a more general approach to these questions has not yet been proposed.

In this work we attempt to lay the foundations to treat the  
class of \emph{coadjoint orbits} of (compact semisimple connected) Lie groups 
in this setting. Coadjoint orbits give examples of a big class of geometrically 
non-trivial symplectic manifolds, which under certain assumptions on the Lie 
group can be equipped with a complex structure. One of the first interesting 
examples is the 2-sphere with its $\group{SO}(3)$-invariant symplectic 
structure, for which many of the currently available methods in the field fail 
\cite{rieffel:1998a}. 
Coadjoint orbits have proven useful in many areas of mathematics and 
physics, e.g.~through their relation to unitary representations of Lie groups 
given by Kirillov's orbit method \cite{kirillov}.

The starting point for the present work are the constructions of star 
products of Wick type due to Karabegov \cite{karabegov:1998c, karabegov:1999a} 
and Alekseev-Lachowska \cite{alekseev.lachowska:2005a}. Earlier constructions 
of star products have been obtained in \cite{moreno.ortega-navarro:1983b, 
cahen.gutt.rawnsley:1990a, cahen.gutt.rawnsley:1993a, 
cahen.gutt.rawnsley:1994a, cahen.gutt.rawnsley:1995a,
bordemann.brischle.emmrich.waldmann:1996a, 
bordemann.brischle.emmrich.waldmann:1996b}.

Recall the following 
definition of star products of Wick type \cite{bordemann.waldmann:1997a}:
A formal star product on a K\"ahler manifold $M$ of the form 
\begin{equation} 
f * g = \sum_{i=0}^\infty \formParam^i C_i(f,g)
\end{equation}
where the $C_i$ are bidifferential 
operators on $M$, is said to be of \emph{Wick type} if when restricted to 
any local holomorphic coordinate system, the $C_i$ contain only holomorphic 
derivatives in the first variable and only antiholomorphic derivatives in the 
second one.
Similarly, a star product of \emph{anti-Wick type} contains only 
antiholomorphic derivatives in the first variable and holomorphic derivatives 
in the second one. Both are also known as star products with separation of 
variables \cite{karabegov:1996a}. Note that both Karabegov and 
Alekseev-Lachowska define 
products in the non-formal setting for certain functions with rational 
dependence on $\hbar$. Then they obtain a formal product by considering 
the pointwise asymptotic expansion around $\hbar = 0$. To the best of our 
knowledge, it 
only 
appeared in the literature that the formal expansions agree, see 
\cite{karabegov:1999a, calaque}, where the Karabegov class is calculated, and 
this class classifies star products of Wick type uniquely 
\cite{karabegov:1996a}. Our first result, 
\autoref{theo:alekseev:constructionsagree} says that the constructions agree on 
the nose:
\begin{maintheorem}\label{maintheo:i}
	Let a coadjoint orbit of a compact connected
	semisimple Lie group $K$ with a fixed complex structure be given. Then
	Ka\-ra\-be\-gov's star product $*_\hbar$ for the trivial Karabegov class
	agrees with the product $*'_\hbar$ defined by Alekseev-Lachowska 
	on the $K$-finite functions whenever $\hbar$ is different from the
	countably many poles.
\end{maintheorem}

The spirit of the constructions is, however, rather different. Karabegov 
deforms the Gutt star product \cite{gutt:1983a} on the dual of a Lie algebra in 
such a way 
that it becomes tangential to the coadjoint orbits and can actually be 
restricted. On the other hand Alekseev-Lachowska construct a formula that looks 
like a Drinfel'd twist. Note also that the construction of Alekseev-Lachowska 
still works for non-compact semisimple Lie groups as long as the orbit contains 
a semisimple element. However, we 
are only interested in the compact case in this paper.

Given a non-formal $K$-invariant Karabegov class with a rational dependence on 
$\hbar$, we use Ka\-ra\-be\-gov's construction to obtain a star product with 
rational 
dependence on $\hbar$ on the polynomials. The formal Karabegov class of this 
product is the formal expansion of the above chosen class.
At the special 
values of the poles, the product is still defined on a 
subspace of the polynomials, giving a possibly finite dimensional algebra. As 
already noticed by Karabegov \cite{karabegov:1999a}, for the example of the 
2-sphere these algebras are exactly the ones obtained via Berezin-Toeplitz 
quantization \cite{cahen.gutt.rawnsley:1990a, cahen.gutt.rawnsley:1993a, 
cahen.gutt.rawnsley:1994a, cahen.gutt.rawnsley:1995a}.

We then restrict to the special case of the 2-sphere and a generic $\hbar \in 
\mathbb C$, meaning that the star product is defined on the whole polynomial 
algebra. We attack the problem of enlarging these algebras by constructing a 
locally convex topology in which the product is continuous. Since the algebra 
of polynomials on the 2-sphere is isomorphic to a quotient of the symmetric 
algebra over a 3-dimensional vector space, we can look at the quotient topology 
induced by the $T_R$-topology used in e.g.~\cite{waldmann:2014a}, that we call 
the \emph{\quotientTR}.

\begin{maintheorem}\label{maintheo:ii}
	The star product $*_\hbar$ on the 2-sphere is continuous
	with respect to the \quotientTR\ for $R \geq 0$ 
	if $\frac 1 \hbar \notin \mathbb N$.
\end{maintheorem}
Working with the Alekseev-Lachowska construction it is even more natural to 
consider a topology defined in a slightly different way, that is however 
equivalent to the \quotientTR.

Having found such a topology, we 
complete the algebra to a Fr\'echet algebra. Note in particular that the 
\quotientTR\ is independent of $\hbar$, so the underlying topological vector 
space of the Fr\'echet algebras is the same. Surprisingly, the completed 
algebras are 
isomorphic to the algebras obtained in \cite{shoetz} for the hyperbolic disc.
However, their *-involutions differ. In a follow up paper we will 
investigate their relations in detail.

In future research, we would also like to extend our convergence results to 
more coadjoint orbits. The question of convergence relates to the coefficients 
occurring in the twist-like formula of Alekseev-Lachowska. These coefficients 
can be obtained in a purely Lie algebraic framework, so that we hope to answer 
questions about convergence of the star product by looking at properties of the 
Lie algebra.

This paper is partially based on the master thesis \cite{philipp:thesis}.
It is organized as follows. In \autoref{sec:preliminaries} we recall 
all the basic notions concerning coadjoint orbits and the construction of 
complex structures that are needed throughout 
the paper. 
Furthermore, we recall the constructions of star products of Wick type on 
coadjoint orbits of compact semisimple and connected Lie groups 
obtained by Alekseev-Lachowska and Karabegov, respectively.
\autoref{sec:alekseev:comparison} contains the comparison between the two 
constructions and we prove \autoref{maintheo:i}.
\autoref{sec:convergence} contains first convergence results for formal star 
products of Wick-type on coadjoint orbits. Here, we prove that the star product 
is defined on the subalgebra of polynomials on the orbit. 
Finally, \autoref{sec:continuity}
contains the main results of this paper. We discuss the particular case of the 
2-sphere interpreted as a coadjoint orbit of $\SU$. We introduce the 
\quotientTR\ and \reductionTR\ and prove that they are equivalent. Furthermore 
we 
prove \autoref{maintheo:ii}.

\subsubsection*{Notation}
We use $\Cinfty(M)$ for real valued smooth functions on a manifold $M$ and 
$\Cinfty_{\mathbb C}(M)$ for complex valued smooth functions. If $M$ is an 
immersed submanifold of a vector space $V$, we denote the restriction of 
complex valued polynomials on $V$ to $M$ by $\Pol(M)$. Moreover, $G$ and $K$ 
denote Lie 
groups with Lie algebras $\lie g$ and $\lie k$, respectively. We always 
use $\formParam$ to denote a formal parameter, $\hbar$ is a complex 
number.

\subsubsection*{Acknowledgements}
We would like thank Matthias Sch\"otz for valuable discussions. The second 
author is supported by the Danish National Research Foundation through the 
Centre for Symmetry 
and Deformation (DNRF92).

\section{Preliminaries}
\label{sec:preliminaries}

In this section we collect some known results on the construction of
star products on coadjoint orbits which will be used in the sequel.
For this purpose, we first recall some basic
tools concerning coadjoint orbits which can be found e.g.~in
\cite{marsden.ratiu:1999a}.

\subsection{Coadjoint orbits and complex structures}
\label{subsec:coadjointorbits}

Let $G$ be a real Lie group with Lie algebra $\lie g$ and denote by
$\coorb{\xi}$ the \emph{coadjoint orbit} of $G$ through $\xi \in
\lie g^*$ and by $\orb{\mu}$ the \emph{adjoint orbit} of $G$ through
$\mu \in \lie g$. Recall that $\coorb \xi$ has a smooth structure,
which makes $\coorb{\xi}$ an immersed submanifold of $\lie g^*$ and
an embedded submanifold if $G$ is compact.
The coadjoint orbits $\coorb{\xi}$ can be endowed with a
$G$-in\-variant symplectic form $\omega_{\coorb{\xi}}$, called the
\emph{Kirillov-Kostant-Souriau (KKS) form}.
\begin{lemma}
  \label{lemma:adjoint=coadjoint}
  Let $G$ be a connected semisimple Lie group. Then:
  \begin{lemmalist}
  \item The adjoint and coadjoint orbits are diffeomorphic via the
  Killing form $B$ of the Lie algebra $\lie g$ of $G$.
  \item Adjoint orbits of $G$ also carry a $G$-invariant symplectic
    form $\omega_{\orb{\mu}}$.
\end{lemmalist}
\end{lemma}
In the following we briefly recall the construction of complex
structures on adjoint orbits of connected compact semisimple Lie
groups (see among others \cite{bordemann.forger.roemer:1986a} for
details, also in the non-compact case), which will be crucial in the
construction of star products.

Let us consider a connected, compact and semisimple Lie group $K$ and
let $\orb{\mu}$ be its adjoint orbit through $\mu \in \lie k$, which
is diffeomorphic to the quotient $K / K^\mu$ of $K$ by the stabilizer
of $\mu$.
Let $\lie k$ and $\lie k^\mu \subseteq \lie k$ be the Lie algebras of
$K$ and $K^\mu$, respectively, and denote their complexifications by
$\lie g$ and $\lie g^\mu$. Consider the orthogonal complement $\lie m$
of $\lie g^\mu$ under the Killing form $B$, for which we have $\lie m
\cong \lie g / \lie g^\mu$ as vector spaces.  Furthermore, let $\lie
t$ be a maximally commutative Lie subalgebra of $\lie k^\mu$
containing $\mu$. It is easy to see that $\lie t$ is also maximally
commutative in $\lie k$, hence its complexification $\lie h$ is a
Cartan subalgebra of $\lie g$ containing $\mu$.  Since $\mu$ commutes
with $\lie h$ and acts as multiplication by $\alpha(\mu)$ on a root
space $\lie g_\alpha$, the subalgebra $\lie g^\mu$ is a sum of $\lie
h$ and the root spaces $\lie g_\alpha$ for which $\alpha(\mu) = 0$.
Let us denote
\begin{equation}
  \label{eq:rootspaces}
  \Delta' 
  =
  \{ 
  \alpha\in \Delta \vert \alpha(\mu) = 0 
  \}
  \quad\text{and}\quad
    \smash{\hat\Delta}
  =
  \{ 
  \alpha \in \Delta \vert \alpha(\mu) \neq 0 
  \} \glpunkt
\end{equation}
In particular, we have
\begin{equation} \label{eq:preliminaries:complexstructure:kmuandm}
  \lie g^\mu = \lie h \oplus \bigoplus_{\alpha \in \Delta'} 
  \lie g_\alpha \quad\text{and}\quad
  \lie m = \bigoplus_{\alpha\in\hat{\Delta}} \lie g_{\alpha} \glpunkt
\end{equation}
Recall that a subset $\hat\Delta^+ \subseteq \hat{\Delta}$ is said to
define an \emph{invariant ordering} on $\hat\Delta$ if for
$\hat{\Delta}^- := - \hat{\Delta}^+$ it satisfies the following
identities:
\begin{gather}
    \hat{\Delta}^+ \cup \hat{\Delta}^- 
    = 
    \hat{\Delta} \quad\text{and}\quad
    \hat{\Delta}^+ \cap \hat{\Delta}^- 
    = 
    \emptyset \glkomma 
    \label{eq:preliminaries:complexstructure:invariantordering:i}\\
    \alpha,     \beta \in\hat{\Delta}^+ ,
    \alpha + \beta \in \hat{\Delta} 
    \implies \alpha+\beta\in\hat{\Delta}^+  
    \glkomma \label{eq:preliminaries:complexstructure:invariantordering:ii}\\
    \alpha 
    \in \hat{\Delta}^+, \gamma \in \Delta', \alpha + \gamma \in 
    \Delta 
    \implies \alpha + \gamma \in \hat{\Delta}^+  \glpunkt 
    \label{eq:preliminaries:complexstructure:invariantordering:iii}
\end{gather}
The following theorem can be found e.g. in
\cite{bordemann.forger.roemer:1986a}.
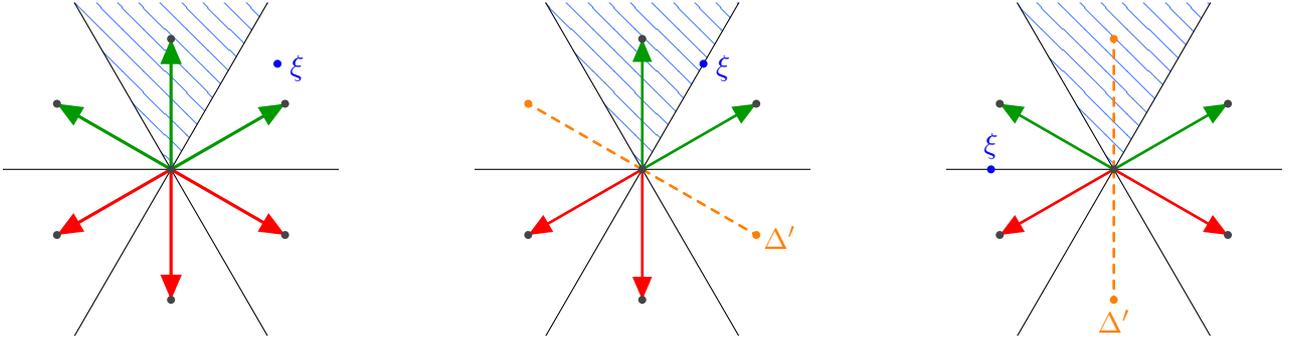
\begin{figure}
	\definecolor{qqzzqq}{rgb}{0,0.6,0}
	\definecolor{qqqqff}{rgb}{0,0,1}
	\definecolor{wwzzff}{rgb}{0.3,0.5,1}
	\definecolor{ffqqqq}{rgb}{1,0,0}
	\definecolor{uuuuuu}{rgb}{0.27,0.27,0.27}
	\begin{tikzpicture}[line cap=round,line join=round,>=triangle 
	45,x=1.0cm,y=1.0cm, scale=1.0]
	\clip(-2.2,-2.2) rectangle (2.2,2.2);
	\fill[pattern=custom north west lines,hatchcolor=wwzzff,hatchspread=8pt] 
	(-4.47,7.74) -- (0,0) -- 
	(4.42,7.66) -- cycle; 
	\draw [domain=-9.1:9.99] plot(\x,{(-0-0*\x)/1});
	\draw [domain=-9.1:9.99] plot(\x,{(-0-0.87*\x)/-0.5});
	\draw [domain=-9.1:9.99] plot(\x,{(-0--0.87*\x)/-0.5});
	\draw [->,line width=1.2pt,color=ffqqqq] (0,0) -- (1.5,-0.87);
	\draw [->,line width=1.2pt,color=ffqqqq] (0,0) -- (0,-1.73);
	\draw [->,line width=1.2pt,color=ffqqqq] (0,0) -- (-1.5,-0.87);
	\draw [->,line width=1.2pt,color=qqzzqq] (0,0) -- (1.5,0.87);
	\draw [->,line width=1.2pt,color=qqzzqq] (0,0) -- (-1.5,0.87);
	\draw [->,line width=1.2pt,color=qqzzqq] (0,0) -- (0,1.73);
	\fill [color=blue] (1.4,1.4) circle (1.5pt);
	\draw [color=blue](1.43,1.65) node[anchor=north west] {$\xi$};
	\begin{scriptsize}
	\fill [color=uuuuuu] (0,0) circle (1.5pt);
	\fill [color=uuuuuu] (0,-1.73) circle (1.5pt);
	\fill [color=uuuuuu] (-1.5,-0.87) circle (1.5pt);
	\fill [color=uuuuuu] (-1.5,0.87) circle (1.5pt);
	\fill [color=uuuuuu] (0,1.73) circle (1.5pt);
	\fill [color=uuuuuu] (1.5,0.87) circle (1.5pt);
	\fill [color=uuuuuu] (1.5,-0.87) circle (1.5pt);
	\end{scriptsize}
	\end{tikzpicture}
	\hfill
	\begin{tikzpicture}[line cap=round,line join=round,>=triangle 
	45,x=1.0cm,y=1.0cm, scale=1.0]
	\clip(-2.2,-2.2) rectangle (2.2,2.2);
	\fill[pattern=custom north west lines,hatchcolor=wwzzff,hatchspread=8pt] 
	(-4.47,7.74) -- (0,0) -- 
	(4.42,7.66) -- cycle; 
	\draw [domain=-9.1:9.99] plot(\x,{(-0-0*\x)/1});
	\draw [domain=-9.1:9.99] plot(\x,{(-0-0.87*\x)/-0.5});
	\draw [domain=-9.1:9.99] plot(\x,{(-0--0.87*\x)/-0.5});
	\draw [line width=1pt,color=orange,dashed] (0,0) -- (1.5,-0.87);
	\draw [->,line width=1pt,color=ffqqqq] (0,0) -- (0,-1.73);
	\draw [->,line width=1pt,color=ffqqqq] (0,0) -- (-1.5,-0.87);
	\draw [->,line width=1pt,color=qqzzqq] (0,0) -- (1.5,0.87);
	\draw [line width=1pt,color=orange,dashed] (0,0) -- (-1.5,0.87);
	\draw [->,line width=1pt,color=qqzzqq] (0,0) -- (0,1.73);
	\fill [color=blue] (0.805,1.4) circle (1.5pt);
	\draw [color=blue](0.835,1.65) node[anchor=north west] {$\xi$};
	\draw [color=orange](2.15,-0.9) node[anchor=east]{$\Delta'$};
	\begin{scriptsize}
	\fill [color=uuuuuu] (0,0) circle (1.5pt);
	\fill [color=uuuuuu] (0,-1.73) circle (1.5pt);
	\fill [color=uuuuuu] (-1.5,-0.87) circle (1.5pt);
	\fill [color=orange] (-1.5,0.87) circle (1.5pt);
	\fill [color=uuuuuu] (0,1.73) circle (1.5pt);
	\fill [color=uuuuuu] (1.5,0.87) circle (1.5pt);
	\fill [color=orange] (1.5,-0.87) circle (1.5pt);
	\end{scriptsize}
	\end{tikzpicture}
	\hfill
	\begin{tikzpicture}[line cap=round,line join=round,>=triangle 
	45,x=1.0cm,y=1.0cm, scale=1.0]
	\clip(-2.2,-2.2) rectangle (2.2,2.2);
	\fill[pattern=custom north west lines,hatchcolor=wwzzff,hatchspread=8pt] 
	(-4.47,7.74) -- (0,0) -- 
	(4.42,7.66) -- cycle; 
	\draw [domain=-9.1:9.99] plot(\x,{(-0-0*\x)/1});
	\draw [domain=-9.1:9.99] plot(\x,{(-0-0.87*\x)/-0.5});
	\draw [domain=-9.1:9.99] plot(\x,{(-0--0.87*\x)/-0.5});
	\draw [line width=1pt,color=orange,dashed] (0,0) -- (0,1.73);
	\draw [->,line width=1pt,color=ffqqqq] (0,0) -- (-1.5,-0.87);
	\draw [->,line width=1pt,color=qqzzqq] (0,0) -- (1.5,0.87);
	\draw [->,line width=1pt,color=ffqqqq] (0,0) -- (1.5,-0.87);
	\draw [line width=1pt,color=orange,dashed] (0,0) -- (0,-1.73);
	\draw [->,line width=1pt,color=qqzzqq] (0,0) -- (-1.5,0.87);
	\fill [color=blue] (-1.615,0) circle (1.5pt);
	\draw [color=blue](-1.615,0) node[anchor=south] {$\xi$};
	\draw [color=orange](0,-1.73) node[anchor=north]{$\Delta'$};
	\begin{scriptsize}
	\fill [color=uuuuuu] (0,0) circle (1.5pt);
	\fill [color=orange] (0,-1.73) circle (1.5pt);
	\fill [color=uuuuuu] (-1.5,-0.87) circle (1.5pt);
	\fill [color=uuuuuu] (-1.5,0.87) circle (1.5pt);
	\fill [color=orange] (0,1.73) circle (1.5pt);
	\fill [color=uuuuuu] (1.5,0.87) circle (1.5pt);
	\fill [color=uuuuuu] (1.5,-0.87) circle (1.5pt);
	\end{scriptsize}
	\end{tikzpicture}
	\caption{Coadjoint orbits of $\SUN 3$. The left example shows a 
		regular orbit of dimension 6, the other two examples show the same 
		non-regular orbit of dimension 4 which is diffeomorphic to 
		$\cpn 2$. The figures show $\lie h^*$ with
		roots in $\hat\Delta^+$ drawn in green and roots in $\hat\Delta^-$ in 
		red. 
		The ordering in the right example is not invariant, the other two orderings 
		are invariant.}
\end{figure}
\begin{theorem}
  \label{theo:preliminaries:complexstructure:bijectiontoinvariantorderings}
  $K$-invariant complex structures on the adjoint orbit $\orb \mu$ are
  in bijection with invariant orderings on $\hat{\Delta}$.
\end{theorem}
In fact, we note that the complexified tangent space $T_\mu^{\mathbb
  C} \orb\mu$ can be identified with the span of $\lie g_\alpha$ for
$\alpha \in \hat\Delta$ via the map
\begin{equation}\label{eq:tangentspaceidentification}
  \langle \lie g_\alpha \mid \alpha \in \hat\Delta \rangle \ni X \longmapsto 
  X_{\orb\mu}(\mu) := \partt \Ad_{\exp(t X)} \mu \in T_\mu^{\mathbb C} \orb \mu 
  \glpunkt 
\end{equation} 
Given an invariant ordering, the corresponding complex structure is
given by multiplication by $\I$ on the $\lie g_\alpha$ with $\alpha
\in \hat\Delta^+$ and by multiplication by $-\I$ on $\lie
g_{-\alpha}$.  Invariant orderings always exist as we can choose any
linear functional $\ell$ on $\lie h^*$ that takes real values on all
roots, vanishes on $\Delta'$ and does not vanish on
$\hat{\Delta}$. Then we get an invariant ordering by letting all
$\alpha \in \hat\Delta^+$ with $\ell(\alpha) > 0$ be positive.

For all $\alpha \in \Delta$ we choose elements $X_\alpha \in \lie
g_\alpha$, $H_\alpha \in \lie h$ and $Y_\alpha \in \lie g_{-\alpha}$ satisfying 
$[X_\alpha, Y_\alpha] = H_\alpha$, $[H_\alpha, X_\alpha] = 2 X_\alpha$ and 
$[H_\alpha, Y_\alpha] = Y_\alpha$ so
that they span subalgebras isomorphic to $\slc$, as can be seen e.g. in
\cite[Chapter II.4, (2.26)]{knapp:lieAlgs}.
As a consequence of the above theorem we obtain the following proposition.
\begin{proposition}
  \label{theo:preliminaries:complexstructure:uniquekaehlerstructure}
  Let $\orb \mu$ be an adjoint orbit of a compact connected semisimple
  Lie group $K$ and $\lambda := - \I B(\mu, \cdot) \in \lie h^*$. Then:
  \begin{propositionlist}
  \item There is a unique $K$-invariant complex structure
    $J_{\textnormal{Kähler}}$ on $\orb{\mu}$ such that $\orb\mu$
    endowed with the KKS form $\omega_{\orb\mu}$, the complex structure
    $J_{\textnormal{Kähler}}$ and the corresponding Riemannian metric
    $g(u,v) := \omega_{\orb\mu}(u,J_{\textnormal{Kähler}} v)$ becomes
    a Kähler manifold.
  \item The complex structure $J_{\textnormal{Kähler}}$ corresponds
    to an invariant ordering on $\smash{\hat{\Delta}}$ for which
    $\smash{\alpha \in \hat \Delta}$ is positive if and only if
    $\lambda(H_\alpha) < 0$.
   \end{propositionlist}
\end{proposition}

\subsection{A construction by Karabegov}
\label{sec:karabegov}

We now recall the construction of star products of Wick type on
coadjoint orbits obtained in \cite{karabegov:1999a} (more details on the
original construction can also be found in \cite{philipp:thesis}).
Notice that, due to the isomorphism between the adjoint and coadjoint
orbits stated in \autoref{lemma:adjoint=coadjoint}, we can
equivalently apply this construction for adjoint orbits.

Consider a compact connected semisimple Lie group $K$ and an adjoint
orbit $\orb{\mu}$, with $\mu \in \lie k$.  Given a complex $K$-invariant 
structure
$J$ on $\orb{\mu}$ one can construct two commuting complex conjugate
representations of the Lie algebra $\lie k$ as differential operators
on $\orb{\mu}$. These representations can be extended to complex
algebra representations of the corresponding complexified universal
enveloping algebra $\univ(\lie g)$.
Thus, composing those algebra representations with the evaluation at
the constant function $\Unit \in \Pol(\orb{\mu})$ we obtain maps
\begin{equation} \label{eq:karabegov:basics}
  L, R : \univ(\lie g) \to \Pol(\orb{\mu}) ,
\end{equation}
allowing us to push forward the non-commutative product from
$\univ(\lie g)$ to $\mathsf{im}(L)\subseteq\Pol(\orb{\mu})$, after
quotienting out the kernel which turns out to be an ideal.
Finally, introducing a parameter $\hbar \in \mathbb C$ in the
construction of the representations gives rise to a star product
of (anti-)Wick type.

In order to construct the representations of $\lie k$ mentioned above,
one needs to introduce a new class of maps on $\orb{\mu}$. 
\begin{definition}[\boldmath $K$-equivariant family]
  A $K$-equivariant family on $\orb{\mu}$ is defined as a map $\lie k
  \to \Cinfty(\orb{\mu})$, $X \mapsto f_X$, which is linear and
  $K$-equivariant with respect to the adjoint action of $K$ on $\lie
  k$ and the shift action on $\Cinfty(\orb\mu)$ given by $k \acts f
  (x) = f(\Ad_{k^{-1}} x)$ for $f \in \Cinfty(\orb\mu)$, $k \in K$ and
  $x \in \orb \mu$.
\end{definition}
$K$-equivariant families can be characterized as follows.
\begin{lemma}\label{lemma:karabegov:kequivfctsandmomentummaps}
  $K$-equivariant families on $\orb{\mu}$ are in one-to-one
  correspondence with $K$-equivariant maps $\gamma: \orb{\mu} \to \lie
  k^*$ via $f_X(x) = \langle \gamma(x) , X \rangle$.
\end{lemma}
Fixing any $x \in \orb\mu$ the map $\gamma \mapsto \gamma(x)$ defines
a bijection between $K$-equivariant maps $\gamma : \orb{\mu} \to \lie
k^*$ and $K^x$-invariant elements of $\lie k^*$. In particular, the
space of $K$-equivariant maps $\gamma: \orb\mu \to \lie k^*$ is a
finite dimensional vector space.

Now consider a $K$-invariant complex structure $J$ on $\orb{\mu}$. 
For $X \in
\lie k$ we can decompose the fundamental vector field $v_X(x) :=
(-X)_{\orb{\mu}}(x) = \parttsmall \exp(-t X) \acts x$ into its $(1,0)$
and $(0,1)$ components with respect to $J$, i.e.
\begin{equation}
  v_X = \xi_X + \eta_X
  \quad \text{with} \quad \xi_X \in \Secinfty(T^{(1,0)} \orb{\mu})
  \quad \text{and}\quad \eta_X \in \Secinfty(T^{(0,1)} \orb{\mu}) 
  \glpunkt
\end{equation}
Note that $v_X \in \Secinfty (T M) \subseteq \Secinfty(T^{\mathbb C}
M)$ is real by definition, hence $\overline{\xi_X} =
\eta_X$. Furthermore, since the action of $K$ is holomorphic the
vector field $\xi_X$ is holomorphic and $\eta_X$ is antiholomorphic.
\begin{proposition}	
  \label{theo:karabegov:repsfromkequivfamily}
  Any $K$-equivariant family $f_X$ on $\orb{\mu}$ defines
  $K$-equivariant commuting complex conjugate representations $\opl,
  \opr: \lie k \to \Diffop(\orb{\mu})$, $X \mapsto \opl_X := \opl(X) =
  \xi_X - \I f_X$ and $X \mapsto \opr_X := \opr(X) = \eta_X + \I f_X$.
\end{proposition}
Here, by commuting we mean that $\opl_X$ and $\opr_Y$ commute for all
$X,Y \in \lie k$.
For a proof of \autoref{theo:karabegov:repsfromkequivfamily}
see \cite{karabegov:1999a}.

Given a $K$-equivariant family on $\orb{\mu}$ we extend the
representations defined in
\autoref{theo:karabegov:repsfromkequivfamily}, $X \mapsto \opl_X$ and
$X \mapsto \opr_X$ of $\lie k$ to complex algebra representations of
the complexified universal enveloping algebra $\univ(\lie g)$.
Let $S$ be the antipode of $\univ(\lie g)$ and $\Unit \in
\Cinfty(\orb{\mu})$ the constant function.
\begin{lemma}\label{lemma:karabegovII:starProdukt:RisLwithAntipode}
  Let $\Opl(u) = \opl(u) \Unit$ and $\Opr(u) = \opr(u) \Unit$.
  \begin{lemmalist}
  \item We have $\Opl( u) = \Opr(S u)$ for all $u \in \univ(\lie g)$.
  \item $\ker_L$ is a two sided ideal in $\univ(\lie g)$.
  \end{lemmalist}
\end{lemma}
\begin{proof}
  For the first item see \cite{karabegov:1999a}. Since $u \mapsto \opl(u)$
  is a homomorphism $\ker_L$ is a left ideal. Using the first claim
  and noting that $S$ is an anti-homomorphism concludes the proof.
\end{proof}
It is important that the image of $L$ and $R$ is actually contained in
$\Pol(\orb\mu) \subseteq \Cinfty_{\mathbb C}(\orb\mu)$. To see this, we note 
that
the polynomials on $\orb\mu$ coincide with the $K$-finite elements of
$\Cinfty_{\mathbb C}(\orb\mu)$, i.e.~the smooth functions whose shifts span 
finite dimensional subspaces, \cite[Lemma 15]{karabegov:1998c}. Using the
definition of a $K$-equivariant family, it is easy to see that all
elements of $\mathsf{im}(L)$ and $\mathsf{im}(R)$ are indeed
$K$-finite.

Now we introduce a parameter $\hbar$ into the construction. 
Suppose that the 
$K$-equivariant family
$\smash{\withhbar{f_\bullet}}$ on $\orb{\mu}$ depends rationally on
$\hbar$ and is regular at $\hbar = 0$. In other words, we suppose that
\begin{equation}
  \withhbar{f_\bullet} 
  = 
  \sum_j b_j(\hbar) f^{j}_\bullet
\end{equation}
for a finite number of $K$-equivariant families $f^{j}_\bullet$ and
rational functions $b_j$ regular at $\hbar = 0$ (i.e.~having no pole
at $\hbar = 0$). Denote the set of poles of the functions $b_j$ by 
$\polefamily$ and note that it is finite.

Let $\sum_{r \geq 0} \formParam^r f_{\bullet,r}$ be
the formal expansion of $\withhbar{f_\bullet}$ around $\hbar = 0$.  In
the following we obtain a formal star product of Wick type from
$\smash{\withhbar{f_\bullet}}$ and in
\autoref{theorem:karabegovII:classification} it is shown that
$f_{\bullet,r}$ determines the Karabegov form associated to this star
product. In particular $f_{\bullet,0}$ determines the symplectic form
that is deformed.

In order to make sure that $f_{\bullet,r}$ determines the $r$-th order
of the Karabegov form, we let $\oplhs (u)$ with $u \in \univ(\lie g)$
be the operators on $\orb{\mu}$ associated to the $K$-equivariant
family $\smash{\frac 1 \hbar \withhbar{f_\bullet}}$ and let $\Oplhs
(u) = \oplhs(u) \Unit$.
\begin{definition}[\boldmath The algebra $(\defspace, *_\hbar)$]
  For $\hbar \in \mathbb C \setminus (P \cup \lbrace 0 \rbrace)$ we define the 
  algebra $(\defspace, 
  *_\hbar)$ to be the pushforward
  of the algebra $\univ(\lie g) / \ker_{\Oplh}$ to
  $\mathsf{im}(\Oplhs) \subseteq \Pol(\orb\mu)$. We let $(\algebra A_0, *_0)$ 
  be the polynomial algebra on the orbit with its commutative product.
\end{definition}
From \autoref{lemma:karabegov:kequivfctsandmomentummaps} we see that
$f_{\bullet,r}$ corresponds to $K$-equivariant maps $\gamma_r: \orb\mu
\to \lie k^*$ and the image of each of these maps is a coadjoint orbit
$\coorbsign_r$.  Denote by $\omega_{\gamma_r}$ the pull-back of
$\omega_{\coorbsign_r}$ via $\gamma_r$.  From now on we assume that
$f_{\bullet,0}$ is \emph{non-degenerate}, meaning that the associated
two-form $\omega_{\gamma_0}$ is non-degenerate.
\begin{theorem}[Karabegov \cite{karabegov:1999a}] \label{theo:karabegov:starproduct} Let $K$ be a
  connected, compact and semisimple Lie group and $\orb{\mu}$ be its
  adjoint orbit through $\mu \in \lie k$.  Fix a $K$-equivariant
  family $\withhbar{f_\bullet} = \sum_j b_j(\hbar) f^{j}_\bullet $,
  where $b_j$ are rational functions of $\hbar$, regular at $0$.
  \begin{theoremlist}
  \item\label{theo:karabegov:starproduct:itemi} Any $u \in \Pol(\orb
    \mu)$ can be written as a finite sum
    \begin{equation}\label{eq:karabegov:polysinimageofl}
      u =\smash{ \sum_j \hbar^{d(j)} a_j(\hbar) \Oplh(u_j)} \glkomma
    \end{equation}
    with  rational functions $a_j$, regular at $0$ (and depending on $u$), and 
    $u_j \in 
    \univdeg{d(j)}(\lie g)$. Here $\univdeg d (\lie g)$
    denotes the filtration of $\univ(\lie g)$, i.e. $\univdeg d (\lie g)$
    is spanned by products of at most $d$ elements from $\lie g$.
      \item\label{theo:karabegov:starproduct:itemii} Any function $u \in
    \Pol(\orb \mu)$ is an element of the algebra $(\defspace,
    *_\hbar)$ for all but a finite number of values of $\hbar$.
  \item\label{theo:karabegov:starproduct:itemiii} For all but a
    countable number of values of $\hbar$, the elements of the
    algebras $\algebra A_\hbar$ and $\Pol(\orb \mu)$ coincide.
  \item\label{theo:karabegov:starproduct:itemiv} If $u,v \in
    \defspace$ and $u$ is written as in
    \eqref{eq:karabegov:polysinimageofl} then
    \begin{equation} \label{eq:karabegov2:starProductDependingOnHbar}
      u *_\hbar v = \sum_j \hbar^{d(j)} a_j(\hbar) \oplh(u_j) v \glpunkt
    \end{equation}
  \item\label{theo:karabegov:starproduct:itemv} For all $x
    \in\orb\mu$, $u, v \in \Pol(\orb\mu)$ the function $\hbar \mapsto
    (u *_\hbar v)(x)$ is rational in $\hbar$ with finitely many poles
    and regular at 0.
  \item\label{theo:karabegov:starproduct:itemvi} $*_\hbar$ satisfies
    \begin{align}
      u *_\hbar v &\to u v \glkomma \label{eq:karabegov:limit:i}
      \\
      \hbar^{-1}(u *_\hbar v - v *_\hbar u) &\to \I\lbrace 
      u, v \rbrace \label{eq:karabegov:limit:ii}
    \end{align}
    for $\hbar \to 0$ (pointwise), where $\lbrace \argument ,
    \argument \rbrace$ is the Poisson bracket corresponding to the
    symplectic form $\omega_{\gamma_0}$.
  \end{theoremlist}
\end{theorem}
Note that \eqref{eq:karabegov:polysinimageofl} and
\eqref{eq:karabegov2:starProductDependingOnHbar} only hold for $\hbar$
different from the set of poles $P$.
Let us discuss some properties of $*_\hbar$.
\begin{corollary}
  \begin{theoremlist}
  \item $*_\hbar$ is $K$-invariant, i.e. $k \acts(u *_\hbar v) = (k
    \acts u) *_\hbar (k \acts v)$ for all $k \in K$.
  \item For every value of $\hbar$ the operator $*_\hbar$ is
    differential on $\algebra A_\hbar$ if any of its arguments is
    fixed.
  \item For every value of $\hbar$ the star product $*_\hbar$ derives
    the first argument only in antiholomorphic and the second argument
    only in holomorphic directions with respect to the chosen complex
    structure $J$.
  \item $*_\hbar$ is of anti-Wick type with respect to the holomorphic
    structure $J$ on $\orb{\mu}$ (respectively, of Wick type with
    respect to $J_{\textnormal{Kähler}}$ if we choose $J$ opposite to
    $J_{\textnormal{Kähler}}$).
  \item For every value of $\hbar$ and $u \in \algebra A_\hbar$ we
    have $\overline u \in \algebra A_{\overline \hbar}$. In particular
    for $\hbar \in \mathbb R \setminus P$, $\algebra A_\hbar$ is closed under
    complex conjugation.
  \item For every value of $\hbar$ the star product $*_\hbar$ is
    Hermitian, meaning that $\overline{u *_\hbar
      v} = \overline v *_{\overline \hbar} \overline u$.
  \end{theoremlist}
\end{corollary}
Note that we cannot expect $*_\hbar$ to be a bidifferential operator
since it will usually contain derivatives of arbitrary high order if
no argument is fixed.
Let us consider the asymptotic expansion of the star product, which
can be obtained from
\eqref{eq:karabegov2:starProductDependingOnHbar}. Since by
\autoref{theo:karabegov:starproduct}
\refitem{theo:karabegov:starproduct:itemv} $u*_\hbar v(x)$ is a
rational function of $\hbar$ with no pole at zero it has an absolutely
and uniformly convergent Taylor series in a small enough neighbourhood
of $\hbar = 0$. We set
\begin{equation}
  \label{eq:taylorseries}
  u \star v = \sum_{r=0}^\infty \frac 1 {r!} \formParam^r \frac {\D^r} {\D 
  \hbar^r}\At{\hbar = 0} u *_\hbar v
\end{equation}
for $u, v \in \Pol(\orb\mu)$. 
From \eqref{eq:karabegov2:starProductDependingOnHbar} we get that in
order $r$ for all $u \in \Pol(\orb{\mu})$, $u \star {}\cdot{}$ is a
differential operator of order at most $r$ and ${}\cdot{} \star u$ is
also a differential operator of order at most $r$.

As a direct consequence of the previous corollary we can prove the
following statement.
\begin{corollary}
  The asymptotic expansion $\star$ of $*_\hbar$ is a $K$-invariant,
  differential, natural, Hermitian formal star
  product of anti-Wick type with respect to $J$.
\end{corollary}
%
Formal star products of Wick type on a symplectic manifold $(M,
\omega_0)$ are classified by formal closed $(1,1)$-forms, see
\cite{karabegov:1996a, neumaier:2003a}. Indeed for any formal closed
$(1,1)$-form $\omega = \sum_{r=0}^\infty z^r \omega_r$, there is a
unique star product of Wick type on $(M, \omega_0)$ such that for all
$f \in \Cinfty_{\mathbb C}(M)$, $f\at U \mathop{\star\at U} \cdot$ commutes with
all holomorphic functions and the operators $\frac{\partial
  \Phi}{\partial z^k}+ \nu \frac{\partial}{\partial z^k}$ and such
that $\cdot \mathop{\star\at U} f\at U$ commutes with all
antiholomorphic functions and the operators
$\frac{\partial\Phi}{\partial \overline z^k}+ \nu \frac \partial
{\partial \overline z^k}$. Here $U$ is any holomorphic chart with
local coordinates $z^1, \dots, z^n$ and $\Phi$ is a potential for
$\omega$ on $U$. Vice versa, the star product determines $\omega$ with
these properties uniquely.
\begin{theorem}[Karabegov \cite{karabegov:1999a}]
  \label{theorem:karabegovII:classification}
  Let $\orb\mu$ be an adjoint orbit of a compact connected semisimple Lie group 
  and let $J$ be the opposite complex structure of
  $J_{\textnormal{Kähler}}$. The formal Wick type star product
  $\star$ associated to a $K$-equivariant family
  $\smash{\withhbar{f_\bullet}}$ with formal expansion $\sum_{r=0}^\infty z^r
  f_{\bullet,r}$ around $\hbar = 0$ has Karabegov form $\omega =
  \sum_{r=0}^\infty z^r \omega_{\gamma_r}$, where $\omega_{\gamma_r} $
  is the pullback of the KKS symplectic form from $\gamma_r(\orb\mu)$ to 
  $\orb\mu$ via
  the map $\gamma_r$ corresponding to $f_{\bullet,r}$.  In
  particular it deforms the symplectic manifold $(\orb \mu,
  \omega_{\gamma_0})$.
\end{theorem}
Note that $\omega_{\gamma_0}$ coincides with the KKS symplectic form
if we choose $\withhbar{f_\bullet}$ such that $f_{X,0}(x)=B(x,X)$
holds for all $X \in \lie k$ and $x \in \orb\mu \subseteq \lie k$.

\subsection{A construction by Alekseev-Lachowska}
\label{sec:alekseev}

The second construction that we recall is due to Alekseev and
Lachowska \cite{alekseev.lachowska:2005a}. The main idea is to use a certain 
pairing
defined below, that is associated to a decomposition of the Lie
algebra, to define an operator depending rationally on $\hbar$. Its
asymptotic expansion gives again an invariant star product of Wick
type. We want to remark here that the construction really builds on
the splitting described below and this splitting is available also in
some nilpotent or infinite dimensional examples, as well as for adjoint orbits 
of non-compact semisimple Lie groups that contain a semisimple element.

Recall that taking left invariant vector fields $X \mapsto
\leftinv{X}$ extends to an isomorphism between the universal
enveloping algebra $\univ(\lie g_{\mathbb C})$ of the complexification $\lie 
g_{\mathbb C}$ of $\lie g$ and the space of
$G$-invariant differential operators $\Diffop^G(G)$ on a Lie group
$G$.

We are interested in the space of $G$-invariant differential operators
on $G/H$, where $H$ is a closed subgroup of $G$. Any $f \in
\Cinfty_{\mathbb C}(G/H)$ can be extended to a function $\smash{\tilde f\in
  \Cinfty_{\mathbb C}(G)}$ that is invariant under right shifts by $H$. The
function $\leftinv{X_1} \cdots \leftinv{X_n} \tilde f$ is right
$H$-invariant again if and only if $X_1 \cdots X_n \in \univ(\lie g_{\mathbb 
C})$
is invariant under the adjoint action of $H$. Since $\leftinv X$ kills
$\tilde f$ if $X \in \lie h$, this can be used to construct an
isomorphism between $(\univ(\lie g_{\mathbb C})/(\univ(\lie g_{\mathbb C}) 
\cdot \lie h_{\mathbb C}))^H$ and the 
space of $G$-invariant differential operators 
$\smash{\Diffop^G(G/H)}$ on a homogeneous space $G/H$. By the
superscript $H$ we mean the elements invariant under the adjoint
action of $H$. Note that the adjoint action of $H$ on $\univ(\lie g_{\mathbb 
C})$
fixes the left ideal $\univ(\lie g_{\mathbb C}) \cdot \lie h_{\mathbb C}$ and 
is therefore
well-defined on the quotient.

A formal $G$-invariant star product on $G / H$ is defined by a series
of $G$-invariant bidifferential operators on $G/H$, i.e.~an element
\begin{equation} 
  \label{eq:alekseev:defB}
  F 
  \in 
  ((\univ(\lie g_{\mathbb C}) /(\univ(\lie g_{\mathbb C}) \cdot \lie h_{\mathbb 
  C}))^{\tensor 2})^H [[\formParam]] 
  \glkomma
\end{equation}
satisfying some further properties that assure associativity and the
usual requirements for zeroth and first order terms. Here the
$H$-action is the diagonal action of $H$ on the tensor product.

We only treat the case of adjoint orbits of compact semisimple connected Lie 
groups here and set our notation
accordingly. Note however, that the following would still make sense for any
grading of $\lie g$ preserved under the adjoint action of $H$.

From now on we let $G:=K$ be a compact connected semisimple Lie group
and $H:=K^\mu$ the stabilizer of some $\mu \in \lie k$, where $\lie k$
and $\lie k^\mu$ denote the Lie algebras of $K$ and $K^\mu$,
respectively. Denote their complexifications by $\lie g$ and $\lie
g^\mu$. Choose a Cartan subalgebra $\lie h$ of $\lie g$ containing
$\mu$, define roots $\Delta$ and subsets $\Delta'$, $\smash{\hat
  \Delta}$ as in \eqref{eq:rootspaces} and choose an ordering on
$\Delta$ such that the induced ordering on $\hat \Delta$ is invariant. Let
$J$ be the corresponding complex structure.  Define a $\mathbb
Z$-grading of $\lie g$ (the graded components being called $\lie g_n$)
in the following way: firstly, let $\lie g_0 = \lie g^\mu$ and
secondly, let all other elements of the root spaces corresponding to
fundamental roots that are not already in $\lie g_0$ have degree 1,
i.e.~$\smash{\lie g_1 = \bigoplus_{\alpha \in \Sigma \setminus
    \Delta'} \lie g_\alpha}$. This gives a well-defined grading since
the ordering of $\hat \Delta$ is invariant.

Define $\alekseevN + = \oplus_{i \geq 1} \lie g_i$ and $\alekseevN - =
\oplus_{i \leq -1} \lie g_i$ and denote the projection to $\lie g_0$
in the direct sum $\lie g = \lie g_0 \oplus \alekseevN + \oplus
\alekseevN -$ by $\pr$. Let $\lambda: \lie g_0 \to \mathbb C$ be a Lie
algebra homomorphism such that the pairing
\begin{equation}
  \alekseevN + \times \alekseevN - \to \mathbb C : \quad \smash{(u,v) \mapsto
	\lambda(\pr([u,v]))}
\end{equation}
is non-degenerate.
We are mainly interested in the case $\lambda = B(-\I \mu,
\cdot)$. Depending on the context we extend $\lambda$ by zero on
$\alekseevN +$ and $\alekseevN -$, still denoted by $\lambda$.
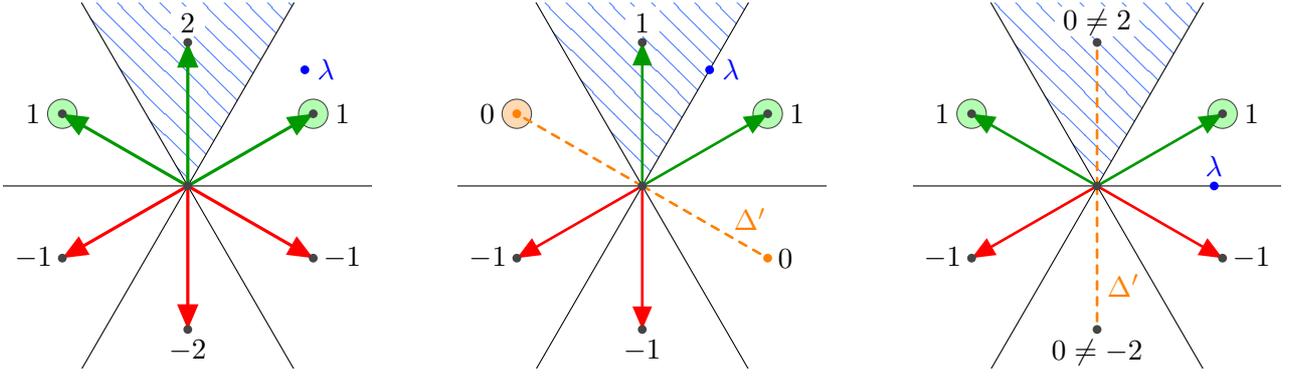
\begin{figure}
	\definecolor{qqzzqq}{rgb}{0,0.6,0}
	\definecolor{qqqqff}{rgb}{0,0,1}
	\definecolor{wwzzff}{rgb}{0.3,0.5,1}
	\definecolor{ffqqqq}{rgb}{1,0,0}
	\definecolor{uuuuuu}{rgb}{0.27,0.27,0.27}
	\begin{tikzpicture}[line cap=round,line join=round,>=triangle 
	45,x=1.0cm,y=1.0cm, scale=1.1]
	\clip(-2.2,-2.2) rectangle (2.2,2.2);
	\fill[pattern=custom north west lines,hatchcolor=wwzzff,hatchspread=8pt] 
	(-4.47,7.74) -- (0,0) -- 
	(4.42,7.66) -- cycle; 
	\draw [domain=-9.1:9.99] plot(\x,{(-0-0*\x)/1});
	\draw [domain=-9.1:9.99] plot(\x,{(-0-0.87*\x)/-0.5});
	\draw [domain=-9.1:9.99] plot(\x,{(-0--0.87*\x)/-0.5});
	\fill [color=green!30] (-1.5,0.87) circle (5pt);
	\fill [color=green!30] (1.5,0.87) circle (5pt);
	\draw [color=uuuuuu] (-1.5,0.87) circle (5pt);
	\draw [color=uuuuuu] (1.5,0.87) circle (5pt);
	\draw [->,line width=1.2pt,color=ffqqqq] (0,0) -- (1.5,-0.87);
	\draw [->,line width=1.2pt,color=ffqqqq] (0,0) -- (0,-1.73);
	\draw [->,line width=1.2pt,color=ffqqqq] (0,0) -- (-1.5,-0.87);
	\draw [->,line width=1.2pt,color=qqzzqq] (0,0) -- (1.5,0.87);
	\draw [->,line width=1.2pt,color=qqzzqq] (0,0) -- (-1.5,0.87);
	\draw [->,line width=1.2pt,color=qqzzqq] (0,0) -- (0,1.73);
	\fill [color=blue] (1.4,1.4) circle (1.5pt);
	\draw [color=blue](1.43,1.65) node[anchor=north west] {$\lambda$};
	\fill[color=white](-0.2,1.81) rectangle (0.2,2.16);
	\draw (-1.5,-0.87) node[anchor=east] {$-1$};
	\draw (1.5,-0.87) node[anchor=west] {$-1$};
	\draw (0,-1.73) node[anchor=north] {$-2$};
	\draw (0,1.73) node[anchor=south] {$2$};	
	\draw (-1.64,0.87) node[anchor=east] {$1$};
	\draw (1.64,0.87) node[anchor=west] {$1$};
	\begin{scriptsize}
	\fill [color=uuuuuu] (0,0) circle (1.5pt);
	\fill [color=uuuuuu] (0,-1.73) circle (1.5pt);
	\fill [color=uuuuuu] (-1.5,-0.87) circle (1.5pt);
	\fill [color=uuuuuu] (-1.5,0.87) circle (1.5pt);
	\fill [color=uuuuuu] (0,1.73) circle (1.5pt);
	\fill [color=uuuuuu] (1.5,0.87) circle (1.5pt);
	\fill [color=uuuuuu] (1.5,-0.87) circle (1.5pt);
	\end{scriptsize}
	\end{tikzpicture}
	\hfill
	\begin{tikzpicture}[line cap=round,line join=round,>=triangle 
	45,x=1.0cm,y=1.0cm, scale=1.1]
	\clip(-2.2,-2.2) rectangle (2.2,2.2);
	\fill[pattern=custom north west lines,hatchcolor=wwzzff,hatchspread=8pt] 
	(-4.47,7.74) -- (0,0) -- 
	(4.42,7.66) -- cycle; 
	\draw [domain=-9.1:9.99] plot(\x,{(-0-0*\x)/1});
	\draw [domain=-9.1:9.99] plot(\x,{(-0-0.87*\x)/-0.5});
	\draw [domain=-9.1:9.99] plot(\x,{(-0--0.87*\x)/-0.5});
	\fill [color=orange!30] (-1.5,0.87) circle (5pt);
	\fill [color=green!30] (1.5,0.87) circle (5pt);
	\draw [color=uuuuuu] (-1.5,0.87) circle (5pt);
	\draw [color=uuuuuu] (1.5,0.87) circle (5pt);
	\draw [line width=1pt,color=orange,dashed] (0,0) -- (1.5,-0.87);
	\draw [->,line width=1pt,color=ffqqqq] (0,0) -- (0,-1.73);
	\draw [->,line width=1pt,color=ffqqqq] (0,0) -- (-1.5,-0.87);
	\draw [->,line width=1pt,color=qqzzqq] (0,0) -- (1.5,0.87);
	\draw [line width=1pt,color=orange,dashed] (0,0) -- (-1.5,0.87);
	\draw [->,line width=1pt,color=qqzzqq] (0,0) -- (0,1.73);
	\fill [color=blue] (0.805,1.4) circle (1.5pt);
	\draw [color=blue](0.835,1.65) node[anchor=north west] {$\lambda$};
	\draw [color=orange](1.6,-0.4) node[anchor=east]{$\Delta'$};
	\fill[color=white](-0.2,1.81) rectangle (0.2,2.16);
	\draw (-1.5,-0.87) node[anchor=east] {$-1$};
	\draw (1.5,-0.87) node[anchor=west] {$0$};
	\draw (0,-1.73) node[anchor=north] {$-1$};
	\draw (0,1.73) node[anchor=south] {$1$};	
	\draw (-1.64,0.87) node[anchor=east] {$0$};
	\draw (1.64,0.87) node[anchor=west] {$1$};
	\begin{scriptsize}
	\fill [color=uuuuuu] (0,0) circle (1.5pt);
	\fill [color=uuuuuu] (0,-1.73) circle (1.5pt);
	\fill [color=uuuuuu] (-1.5,-0.87) circle (1.5pt);
	\fill [color=orange] (-1.5,0.87) circle (1.5pt);
	\fill [color=uuuuuu] (0,1.73) circle (1.5pt);
	\fill [color=uuuuuu] (1.5,0.87) circle (1.5pt);
	\fill [color=orange] (1.5,-0.87) circle (1.5pt);
	\end{scriptsize}
	\end{tikzpicture}
	\hfill
	\begin{tikzpicture}[line cap=round,line join=round,>=triangle 
	45,x=1.0cm,y=1.0cm, scale=1.1]
	\clip(-2.2,-2.2) rectangle (2.2,2.2);
	\fill[pattern=custom north west lines,hatchcolor=wwzzff,hatchspread=8pt] 
	(-4.47,7.74) -- (0,0) -- 
	(4.42,7.66) -- cycle; 
	\draw [domain=-9.1:9.99] plot(\x,{(-0-0*\x)/1});
	\draw [domain=-9.1:9.99] plot(\x,{(-0-0.87*\x)/-0.5});
	\draw [domain=-9.1:9.99] plot(\x,{(-0--0.87*\x)/-0.5});
	\fill [color=green!30] (-1.5,0.87) circle (5pt);
	\fill [color=green!30] (1.5,0.87) circle (5pt);
	\draw [color=uuuuuu] (-1.5,0.87) circle (5pt);
	\draw [color=uuuuuu] (1.5,0.87) circle (5pt);
	\draw [line width=1pt,color=orange,dashed] (0,0) -- (0,-1.73);
	\draw [->,line width=1pt,color=ffqqqq] (0,0) -- (1.5,-0.87);
	\draw [->,line width=1pt,color=ffqqqq] (0,0) -- (-1.5,-0.87);
	\draw [->,line width=1pt,color=qqzzqq] (0,0) -- (1.5,0.87);
	\draw [line width=1pt,color=orange,dashed] (0,0) -- (0,1.73);
	\draw [->,line width=1pt,color=qqzzqq] (0,0) -- (-1.5,0.87);
	\fill [color=blue] (1.4,0) circle (1.5pt);
	\draw [color=blue](1.4,0) node[anchor=south] {$\lambda$};
	\draw [color=orange](0,-1.2) node[anchor=west]{$\Delta'$};
	\fill[color=white](-0.5,1.81) rectangle (0.5,2.16);
	\draw (-1.5,-0.87) node[anchor=east] {$-1$};
	\draw (1.5,-0.87) node[anchor=west] {$-1$};
	\draw (0,-1.73) node[anchor=north] {$0 \neq -2$};
	\draw (0,1.7) node[anchor=south] {$0 \neq 2$};	
	\draw (-1.64,0.87) node[anchor=east] {$1$};
	\draw (1.64,0.87) node[anchor=west] {$1$};
	\begin{scriptsize}
	\fill [color=uuuuuu] (0,0) circle (1.5pt);
	\fill [color=uuuuuu] (0,-1.73) circle (1.5pt);
	\fill [color=uuuuuu] (-1.5,-0.87) circle (1.5pt);
	\fill [color=uuuuuu] (-1.5,0.87) circle (1.5pt);
	\fill [color=uuuuuu] (0,1.73) circle (1.5pt);
	\fill [color=uuuuuu] (1.5,0.87) circle (1.5pt);
	\fill [color=uuuuuu] (1.5,-0.87) circle (1.5pt);
	\end{scriptsize}
	\end{tikzpicture}
	\caption{Illustration of the grading. Fundamental roots are 
		encircled. Positive roots are drawn green and negative ones red if they 
		lie in $\hat{\Delta}$. Roots from 
		$\Delta'$ are drawn with orange dashed lines. The grading is indicated 
		next to each root space. The Cartan 
		subalgebra has always grading $0$. A regular orbit of $\SUN 3$ is shown 
		on the 
		left, the other two pictures are of 
		non-regular orbits, $\lambda = B(-\I \mu, \cdot)$. In the right picture 
		the ordering on $\Delta'$ is 
		not invariant and therefore the grading is not well-defined.}
\end{figure}
For any $\hbar \in \mathbb C \setminus \lbrace 0 \rbrace$ we consider
the pairing (which is related to the Shapovalov pairing of the
corresponding Verma modules)
\begin{equation} 
  \label{eq:alekseev:pairing}
  (\cdot, \cdot)_\hbar 
  : 
  \univ(\alekseevN -) \times \univ(\alekseevN +) 
  \to
  \mathbb C\,,
  \quad
  (y,x) \mapsto (y,x)_\hbar 
  = 
  \evaluatebig{\frac \lambda \hbar}{(S(x)y)_0} 
  \glpunkt 
\end{equation}
Here $(\cdot)_0$ is the projection onto the second summand in
$\univ(\lie g) = (\alekseevN - \univ (\lie g) + \univ (\lie g)
\alekseevN +) \oplus \univ (\lie g_0)$.

It is important to notice that $\univ(\lie g)$ is graded: we can
simply define the grade of an element $X_1 \cdots X_n \in \univ(\lie
g)$ to be the sum of the grades of the $X_i \in \lie g$. This is well
defined, since the ideal generated by $X Y - Y X -[X, Y]$ is
homogeneous.  The graded components of $\univ(\alekseevN -)$ and
$\univ(\alekseevN +)$ are all finite dimensional and the pairing in
\eqref{eq:alekseev:pairing} respects the grading in the sense that for
homogeneous elements $x \in \univ(\alekseevN -)$ and $y \in
\univ(\alekseevN +)$ it is non-zero only if the degrees of $x$ and $y$
add up to zero.  Thus we can choose homogeneous bases of
$\univ(\alekseevN -)$ and $\univ(\alekseevN +)$ such that the pairing
is block diagonal with each block being finite dimensional.  With a
careful analysis of the pairing one can then prove that each block is
invertible with only finitely many exceptions for $\hbar$, see
\cite[Proposition 3.1]{alekseev.lachowska:2005a}.
\begin{proposition}
  The pairing $(\cdot, \cdot)_\hbar$ is non-singular for almost all
  $\hbar \in \mathbb C \setminus \lbrace 0 \rbrace$.
\end{proposition}
Here by non-singular we mean that for all $x \in \univ(\alekseevN -)$
there is $y \in \univ(\alekseevN +)$ with $(x,y)_\hbar \neq
0$. Consequently we can choose bases $\lbrace 1, y^\hbar_1, y^\hbar_2,
\dots \rbrace$ of $\univ(\alekseevN -)$ and $\lbrace 1, x^\hbar_1,
x^\hbar_2, \dots \rbrace$ of $\univ(\alekseevN +)$ ordered by
increasing grading for almost all $\hbar \in \mathbb C$, that are dual
to each other with respect to $(\cdot, \cdot)_\hbar$. Then
\begin{equation} 
  \label{eq:alekseev:FLambda}
  F_\hbar = 1\tensor 1+y^\hbar_1 \tensor x^\hbar_1 + y^\hbar_2 \tensor x^\hbar_2 
  + \dots \in 
  \univ(\alekseevN 
  -) \mathop{\hat{\tensor} }
  \univ(\alekseevN +)
\end{equation}
is well-defined, independent of the bases chosen. Note that of the
infinitely many terms appearing in the formula for $F_\hbar$ only
finitely many lie in a certain grade. Finally, in \cite[Theorem
4.9]{alekseev.lachowska:2005a} the authors prove the following result.
\begin{theorem}[Alekseev-Lachowska]\label{theo:starproduct:alekseev}
  Let $\orb\mu$ be an adjoint orbit of a compact connected semisimple
  Lie group $K$. With the notation introduced above
  \begin{theoremlist}
  \item \label{theo:starproduct:alekseev:i} $F_\hbar$ depends rationally on 
  $\hbar$, with no pole at zero.
  \item \label{theo:starproduct:alekseev:ii} The formal Taylor series expansion 
  of $F_{\hbar}$ around $0$
    gives an element $F \in ((\univ(\lie g) /(\univ(\lie g) \cdot \lie
    g^\mu))^{\tensor 2})^{K^\mu} [[\formParam]]$. The elements
    $F_{\hbar}$ and $F$ satisfy an associativity condition.
  \end{theoremlist}
\end{theorem}

\begin{remark}
Part \refitem{theo:starproduct:alekseev:i} of the previous 
theorem has to be 
understood in the 
	sense that up to a certain degree in $\univ(\alekseevN -) 
	\mathop{\hat\tensor} \univ(\alekseevN +)$ the element $F_\hbar$ is 
	meromorphic. A priori we cannot say that the poles up to arbitrary degree 
	behave nicely, meaning that we cannot exclude the case that the (countable) 
	set of poles has accumulation points different from zero. However, in the 
	example of the 2-sphere the explicit calculations in later sections show 
	that this is not the case, i.e.~for the two sphere 0 is the only 
	accumulation point.
	
Note that the asymptotic expansion is well-defined because 
only finitely many terms of $F_\hbar$ contribute to a certain degree, due to 
the fact that the asymptotic expansion of $y_n^\hbar \tensor x_n^\hbar$ will 
have increasing powers of $\hbar$ as $n \to \infty$. See \cite[Remark 
3.4]{alekseev.lachowska:2005a} for details.
\end{remark}
The next subsection will show that only finitely many terms in
\eqref{eq:alekseev:FLambda} contribute when applying $F_\hbar$ to two
polynomials $p, q \in \Pol(K/K^\mu)$. Thus, the corresponding star
product
\begin{equation}
  p \mathop{*'_\hbar} q 
  :=
  F_\hbar(p,q)
\end{equation}
is well-defined for $\hbar$ different from the poles. If
$\lambda=B(-\I\mu,\cdot)$, the formal star product deforms the the
Kirillov-Kostant-Souriau Poisson structure for adjoint orbits of
connected semisimple compact Lie groups.

%

%
%

\section{The products of Karabegov and Alekseev-Lachowska agree} 
\label{sec:alekseev:comparison}

We here show that the product by Alekseev-Lachowska from
\autoref{theo:starproduct:alekseev} coincides with the star product of
Karabegov defined in \autoref{theo:karabegov:starproduct}.  Note that
$\lambda : \lie g_0 \to \mathbb C$ in the following theorem can be any
Lie algebra homomorphism with $\lambda(\lie g_0 \cap\lie k) \subseteq
\I \mathbb R$, but only for the special choice
$\lambda=B(-\I\mu,\cdot)$ does the star product deform the KKS
symplectic form.

\begin{theorem}
  \label{theo:alekseev:constructionsagree}
  Let $\orb\mu$ be a fixed adjoint orbit of a compact connected
  semisimple Lie group with a fixed complex structure $J$. Then
  Ka\-ra\-be\-gov's star product $*_\hbar$ with respect to the
  $K$-equivariant family defined by $\smash{\withhbar{f_X}(\mu)} =\I 
  \lambda(X)$ agrees
  with the product $*'_\hbar$ defined by Alekseev and Lachowska with
  respect to $\lambda$ whenever $\hbar$ is different from the
  countably many poles.
\end{theorem}
In order to prove \autoref{theo:alekseev:constructionsagree} we
introduce a function $s_\lambda$, already defined in
\cite{karabegov:1998c} by
\begin{equation}
  s_\lambda \colon \univ(\lie g) \to \Cinfty_{\mathbb C}(K)\glkomma
  \quad 
  s_\lambda u(k)=\evaluate{\lambda}{(\Ad_{k^{-1}} u)_0} \glpunkt
\end{equation}
Furthermore we let 
\begin{equation}
  \psi_\mu : K \to K/K^\mu\cong\orb\mu
\end{equation}
be the projection. It is clear that both $s_\lambda$ and $\psi_\mu$
are $K$-equivariant, where $K$ acts under the adjoint action on
$\univ(\lie g)$ and by left translations on $\Cinfty_{\mathbb C}(K)$.
\begin{lemma}
  \label{lemma:comparison:slambda}
  Taking the $K$-equivariant family $\smash{\withhbar{f_X}} : \orb \mu \to 
  \mathbb R$
  defined by $\smash{\withhbar{f_X}}(\mu) = \I \lambda(X)$, we have
  \begin{equation} \label{eq:lemma:comparison:slambda}
    \Oplh(u) \circ \psi_\mu = s_{\lambda/\hbar} u
  \end{equation}
  for all $u \in \univ(\lie g)$.
\end{lemma}
\begin{proof}
  By $K$-equivariance of all involved maps, it suffices to check this
  at the unit $e \in K$.  Hence we need to see that $\smash{\Oplh(u)
    (\mu)} = \evaluate{\lambda/\hbar}{u_0}$ holds.  First, if $X = E -
  \I J E \in \alekseevN +$ with $E \in \lie k \cap \bigcup_{\alpha \in
    \hat \Delta} \lie g_\alpha$ then
  \begin{equation*}
    \oprh(X) 
    = (-X_{\orb\mu})^{(0,1)} + \frac \I \hbar \withhbar{f_X} \glkomma
  \end{equation*}
  and both $-X_{\orb\mu}^{(0,1)}$ and $\smash{\withhbar{f_X}}$ vanish at $\mu$. 
  Similarly, if $Y \in \alekseevN -$ then $\oplh(Y) =
  (-Y_{\orb\mu})^{(1,0)}-\frac \I \hbar \smash{\withhbar{f_Y}}$ vanishes at 
  $\mu$.  It
  suffices to prove the statement for $u$ in canonical form, i.e.~$u =
  Y_1 \cdots Y_n H_1 \cdots H_m X_1 \cdots X_k$ with $Y_i\in\alekseevN
  -$, $X_i\in\alekseevN +$ and $H_i \in \lie g_0$. If $n \geq 1$ then
  \begin{equation*}
    \Oplh (u)(\mu) = \oplh(Y_1) \oplh(Y_2 \cdots Y_n H_1 \cdots H_m 
    X_1 \cdots X_k) \Unit 
    (\mu) = 0 = \evaluatebig{\frac\lambda\hbar}{u_0}
  \end{equation*} 
  and 
  similarly if $k \geq 1$
  \begin{align*}
    \Oplh(u)(\mu )
    &= 
    \Oprh(S(u)) (\mu) 
    \\ 
    &=
    \oprh(-X_k) \oprh(S(Y_1 \cdots Y_n H_1 \cdots H_m X_1 \cdots X_{k-1})) 
    \Unit(\mu)
    \\
    &= 0 \\
    &= 
    \evaluatebig{\frac\lambda\hbar}{u_0}\glpunkt
  \end{align*} 
  Finally, for $n=k=0$ we
  have $v_H (\mu)= 0$ for $H \in \lie g_0$ and therefore $\xi_H (\mu)=
  \eta_H (\mu)= 0$. So
  \begin{align*}
    \Oplh (u)(\mu) &= 
    \frac {-\I} \hbar \withhbar{f_{H_1}} (\mu) \cdots \frac 
    {-\I} \hbar  \withhbar{f_{H_m}} (\mu) 
    \\
    &= 
    \left(\frac {-\I} \hbar \right)^m 
    \I \lambda(H_1) \cdots \I \lambda(H_m)
    \\
    &= 
    \frac 1 {\hbar^m} \lambda(H_1) \cdots 
    \lambda(H_m) 
    \\
    &= \evaluatebig{\frac \lambda \hbar}{u_0} \glpunkt
  \end{align*}
\end{proof}
In the following for $X = E + \I F \in \lie g$ with $E, F \in \lie k$
let 
$X_M \at x:= E_M \at x+ \I F_M \at x\in T_x^{\mathbb C} M$, where $M$ is a
manifold with an action of $K$, e.g.~$K$ or $\orb\mu$.

We enumerate the positive roots by $\alpha_1, \dots, \alpha_N$ and let
$A \in \lbrace 1, \dots, N \rbrace^L$ be a tupel of $L$ elements. Let
$X_{\alpha_i}$, $H_{\alpha_i}$, $Y_{\alpha_i}$ span subalgebras of
$\lie g$ isomorphic to $\lie {sl}_2$ as in
\autoref{subsec:coadjointorbits}.
\begin{lemma}\label{lemma:alekseev:LieDerivatives}
  We have
  \begin{equation} 
    \label{eq:alekseev:LieDerivatives}
    \Lie_{\leftinv{X_A}} s_{\lambda} u(e) = 
    s_\lambda(S(X_A) u)(e) \quad\text{ and }\quad 
    \Lie_{\leftinv{Y_A}} s_{\lambda} u(e) = 
    s_\lambda(u Y_A)(e) \glkomma
  \end{equation}
  where $\Lie_{\leftinv{X_A}}$ is a shorthand for
  $\Lie_{\leftinv{(X_{\alpha_{A(1)}})}} \cdots
  \Lie_{\leftinv{(X_{\alpha_{A(L)}})}}$ and $X_A$ means
  $X_{\alpha_{A(1)}} \cdots X_{\alpha_{A(L)}}$ and similarly for
  $\Lie_{\leftinv{Y_A}}$ and $Y_A$.
\end{lemma}
\begin{proof}
  Left invariant vector fields are the fundamental vector fields of
  the right action of $K$ on $K$ by right multiplications. Since there
  exists a complex Lie group $G$ corresponding to $\lie g$ and all
  involved maps are complex linear, we can also obtain the fundamental
  vector field of $X \in \lie g$ by taking the fundamental vector
  field of the right action of $G$ on $G$. Hence
  \begin{align*}
    \Lie_{\leftinv{X_A}} s_{\lambda} u(e) &= 
    \Lie_{\leftinv{(X_{\alpha_{A(1)}})}} \cdots 
    \Lie_{\leftinv{(X_{\alpha_{A(L)}})}} 
    s_{\lambda} u (e) \\
    &= \parttj 1 \cdots \parttj L s_\lambda u (
    \exp(t_1 X_{\alpha_{A(1)}}) \cdots \exp(t_L X_{\alpha_{A(L)}}) ) \\
    &=  \parttj 1 \cdots \parttj L \evaluatebig{\lambda}{\left(\Ad_{(\exp( t_1 
	  X_{\alpha_{A(1)}}) 
          \cdots 
          \exp(t_L X_{\alpha_{A(L)}}))^{-1}} u \right)_0} \\
    &=  \parttj 1 \cdots \parttj L 
    \evaluatebig{\lambda}{\left(\Ad_{\exp(-t_L 
	  X_{\alpha_{A(L)}})} 
        \cdots \Ad_{\exp(-t_1 
          X_{\alpha_{A(1)}})} u\right)_0} \\
    &=  \evaluatebig{\lambda}{\left(\parttj L \exp\left(-t_L 
	  \ad_{X_{\alpha_{A(L)}}}\right) \cdots 
        \parttj 1  \exp\left(-t_1 \ad_{X_{\alpha_{A(1)}}}\right) 
        u\right)_0} \\
    &= \evaluatebig{\lambda}{\left(\left(- 
    \ad_{X_{\alpha_{A(L)}}}\right)\cdots\left( 
	  -\ad_{X_{\alpha_{A(1)}}}\right) u \right)_0} \\
    &= \evaluatebig{\lambda}{\left((- {X_{\alpha_{A(L)}}})\cdots ( 
        -{X_{\alpha_{A(1)}}}) u \right)_0} \\
    &= \evaluatebig{\lambda}{\left(
        S(X_{\alpha_{A(1)}} \cdots X_{\alpha_{A(L)}} ) u \right)_0} \\
    &= \evaluate{\lambda}{ (
      S(X_{A} ) u)_0} \\
  &= s_\lambda(S(X_A)u)(e) \glpunkt
  \end{align*}
  Here we used that $(u X_{\alpha_i})_0$ vanishes for all $1 \leq i
  \leq N$ and all $u \in \univ(\lie g)$.
  The calculation for $Y$ is similar except that now $(Y_{\alpha_i}
  u)_0 $ vanishes and therefore
  \begin{equation*} 
  \evaluatebig{\lambda}{\left(
  	  \left(- \ad_{Y_{\alpha_{A(L)}}} \right)
  	  \cdots
  	  \left(- \ad_{Y_{\alpha_{A(1)}}} \right) u
  \right)_0} 
  = \evaluatebig{\lambda}{\left(
  	 u Y_{\alpha_{A(1)}} 
  	   \cdots  
       Y_{\alpha_{A(L)}}
  \right)_0}
  = s_\lambda(u Y_A)(e) \glpunkt
  \end{equation*}
\end{proof}
Finally, we can complete the comparison of the star products defined
by Karabegov and Alekseev-Lachowska.
\begin{proofof}[Proof of \autoref{theo:alekseev:constructionsagree}:]
  Fix an adjoint orbit $\orb \mu$. Since both star products are
  $K$-invariant it suffices to check that they agree at $\mu$. Fix two
  functions $u, v \in \Pol(\orb{\mu})$. By
  \autoref{theo:karabegov:starproduct},
  \refitem{theo:karabegov:starproduct:itemi} we have
  \begin{equation*}
    u =\sum_{i=1}^m \hbar^{d(i)} a_i(\hbar) \Oplh(u_i) 
    \quad\text{and}\quad
    v =\sum_{j=1}^\ell \hbar^{e(j)} b_j(\hbar) \Oplh(v_j) \glkomma
  \end{equation*}
  where $a_i$, $b_j$ are rational functions of $\hbar$, regular at $0$
  and $u_i \in \univdeg{d(i)}(\lie g)$, $v_j \in \univdeg{e(j)}(\lie
  g)$. Consequently
  \begin{equation}
    \label{eq:alekseev:proof:i}
      u \circ \psi_\mu = \sum_{i=1}^{m}  \hbar^{d(i)} a_i(\hbar)
      s_{\lambda/\hbar} u_i \glkomma \quad
      v \circ \psi_\mu = \sum_{j=1}^{\ell}  \hbar^{e(j)} b_j(\hbar)
      s_{\lambda/\hbar} v_j 
    \end{equation}
    and
    \begin{align}
      (u *_\hbar v)\circ \psi_\mu 
      &= 
      \sum_{i=1}^{m}\sum_{j=1}^{\ell}  
      \hbar^{d(i)+e(j)} a_i(\hbar)b_j(\hbar) \Oplh(u_i v_j)\circ\psi_\lambda
      \nonumber
      \\
      &= 
      \sum_{i=1}^{m}\sum_{j=1}^{\ell}  
      \hbar^{d(i)+e(j)} a_i(\hbar)b_j(\hbar) 
      s_{\lambda/\hbar} (u_i v_j)\label{eq:alekseev:proof:ii}
  \end{align}
  by \autoref{lemma:comparison:slambda}.  Let $\lbrace y_0^\hbar := 1,
  y_1^\hbar, y_2^\hbar, \dots \rbrace$ and $\lbrace x_0^\hbar :=1, x_1^\hbar,
  x_2^\hbar, \dots \rbrace$ be the bases of $\univ(\alekseevN -)$ and
  $\univ(\alekseevN +)$ used in \autoref{sec:alekseev} to construct
  $F_\hbar$. Enumerate the positive roots as above and write
  $y^\hbar_n = \sum_{p=1}^{m_n} c_{n,p}^\hbar Y_{N_{n,p}}$ for complex
  coefficients $c^\hbar_{n,p}$ and tuples $N_{n,p} \in \lbrace 1,
  \dots, N \rbrace ^{L_{n,p}}$ of length $\smash{L_{n,p}}$. Similarly,
  write $x^\hbar_n = \smash{\sum_{q=1}^{\ell_n} d^\hbar_{n,q}
    X_{M_{n,q}}}$. We drop all sums over $p$ and $q$ below. Then
  \begin{align*} 
    (u *_\hbar v) (\mu) 
    &\textueber{=}{(1)} 
    \sum_{i,j=1}^{m,\ell} 
    a_i(\hbar) b_j(\hbar) \hbar^{d(i)+e(j)} 
    s_{\lambda/\hbar}(u_i v_j) (e)  \\
    &\textueber{=}{(2)}  \sum_{i,j=1}^{m,\ell} \sum_{n=0}^\infty 
    a_i(\hbar) b_j(\hbar) 
    \hbar^{d(i)+e(j)} 
    s_{\lambda/\hbar}(u_i c^\hbar_{n,p} Y_{N_{n,p}})(e) 
    s_{\lambda/\hbar}(S(d^\hbar_{n,q} 
    X_{M_{n,q}}) v_j) 
    (e) \\
    &\textueber{=}{(3)}  \sum_{i,j=1}^{m,\ell} \sum_{n=0}^\infty 
    a_i(\hbar) b_j(\hbar) 
    \hbar^{d(i)+e(j)} 
    c^\hbar_{n,p} \Lie_{\leftinv{(Y_{N_{n,p}})}} s_{\lambda/\hbar}(u_i)(e) 
    d^\hbar_{n,q} \Lie_{\leftinv{(X_{M_{n,q}})}} s_{\lambda/\hbar}(v_j)(e)
    \\
    &\textueber{=}{(4)}  \sum_{n=0}^\infty 
    c^\hbar_{n,p} \Lie_{\leftinv{(Y_{N_{n,p}})}} (u \circ\psi_\lambda) (e) 
    d^\hbar_{n,q} \Lie_{\leftinv{(X_{M_{n,q}})}} (v \circ\psi_\lambda) (e)
    \\
    &\textueber{=}{(5)}  F_{\hbar}(u,v) (\mu)
    =   (u *'_\hbar v) (\mu) \glpunkt
  \end{align*} 
  Here (1) is \eqref{eq:alekseev:proof:ii}, (2) holds because $\lbrace
  y_0 := 1, y^\hbar_1, y^\hbar_2, \dots \rbrace$ and $\lbrace x_0:=1,
  x^\hbar_1, x^\hbar_2, \dots \rbrace$ form dual bases with respect to
  the pairing \eqref{eq:alekseev:pairing}. Note that only finitely
  many terms contribute to the sum over $n$, so there are no
  convergence issues. The argument in (3) is simply
  \autoref{lemma:alekseev:LieDerivatives} and (4) is
  \eqref{eq:alekseev:proof:i}. Finally (5) comes from the
  identification of elements in $((\univ(\lie g) /(\univ(\lie g) \cdot
  \lie g^\mu)^{\tensor 2})^{K^\mu}$ with bidifferential operators.
\end{proofof}
Note that since both $F_\hbar(u,v)$ and $u*_\hbar v$ depend rationally
on $\hbar$ (for fixed polynomials $u$, $v$) the proof implies in
particular that these rational functions agree almost everywhere,
hence must have the same poles.

%
%

\section{First convergence properties}
\label{sec:convergence}
In the following we use the Karabegov description to see that any
formal $K$-invariant star product $\star$ of Wick type with a nice
Karabegov form is the Taylor expansion of a strict product on $\Pol(\orb\mu)$. 
We use the construction of Karabegov from 
\autoref{sec:karabegov} to obtain this strict product. To do this, we first 
need to prove a 
couple of technical lemmas.
 
For two roots $\alpha, \beta \in \Delta$ with $\alpha+\beta\in \Delta$
we define $N_{\alpha,\beta} \in \mathbb{C}$ by $[X_\alpha, X_\beta] =
N_{\alpha, \beta} X_{\alpha + \beta}$. It is known that
$N_{\alpha,\beta} \neq 0$, \cite[Chapter VI.1, Theorem
6.6]{knapp:lieAlgs}. Recall that $v_X := (-X)_{\orb\mu} := \parttsmall 
\Ad_{\exp(-tX)} x$ for $x \in \orb\mu$ and $X \in \lie k$.
\begin{lemma}
\label{lemma:Kinvariantclosed2}
  A $K$-invariant closed 2-form on $\orb{\mu}$ satisfies
  \begin{equation}
    \omega([X_\alpha, X_\beta], X_\gamma) + \omega([X_\beta, X_\gamma], X_\alpha) + \omega([X_\gamma, 
    X_\alpha], X_\beta) = 0\glpunkt
  \end{equation}
\end{lemma}
Here we use the identifications from \eqref{eq:tangentspaceidentification}.
\begin{proof}
  For three fundamental vector fields $(X_\alpha)_{\orb{\mu}}$,
  $(X_\beta)_{\orb{\mu}}$ and $(X_{\gamma})_{\orb{\mu}}$ we calculate
  \begin{align*}
    0 &= \D 
    \omega ((X_\alpha)_{\orb{\mu}},(X_\beta)_{\orb{\mu}} 
    ,(X_{\gamma})_{\orb{\mu}}) \\
    &= 
    (X_\alpha)_{\orb{\mu}} \omega((X_\beta)_{\orb{\mu}},(X_{\gamma})_{\orb{\mu}})
    + 
    (X_\beta)_{\orb{\mu}} \omega((X_{\gamma})_{\orb{\mu}},(X_\alpha)_{\orb{\mu}}) 
    \\
    &\phantom{SPACE}+ 
    (X_{\gamma})_{\orb{\mu}} \omega((X_\alpha)_{\orb{\mu}},(X_\beta)_{\orb{\mu}}) 
    - 
    \omega([(X_\alpha)_{\orb{\mu}},(X_\beta)_{\orb{\mu}}],(X_{\gamma})_{\orb{\mu}}) 
    \\
    &\phantom{SPACE}- 
    \omega([(X_\beta)_{\orb{\mu}},(X_{\gamma})_{\orb{\mu}}],(X_\alpha)_{\orb{\mu}}) 
    - 
    \omega([(X_{\gamma})_{\orb{\mu}},(X_\alpha)_{\orb{\mu}}],(X_\beta)_{\orb{\mu}}) 
    \\
    &= \omega([X_\alpha, X_\beta]_{\orb{\mu}},(X_{\gamma})_{\orb{\mu}}) 
    + 
    \omega([X_\beta,X_{\gamma}]_{\orb{\mu}},(X_\alpha)_{\orb{\mu}}) 
    +
    \omega([X_{\gamma},X_\alpha]_{\orb{\mu}},(X_\beta)_{\orb{\mu}}) \glpunkt
  \end{align*}
  Here we used that since $\omega$ is $K$-invariant, its Lie
  derivative with respect to fundamental vector fields vanishes.
\end{proof}
\begin{lemma}
  \label{lemma:Kinvariantclosed11}
  Any $K$-invariant closed real $(1,1)$-form $\omega$ on a coadjoint
  orbit $\orb\mu$ can be written as $\omega = \omega_\gamma$ for some
  $K$-equivariant map $\gamma : \orb\mu \to \coorbsign \subseteq \lie
  k^*$.
\end{lemma}
\begin{proof}
  By $K$-equivariance it suffices to check the claim at $\mu$, by
  which we mean that there is a $K^\mu$-invariant element $\gamma(\mu)
  \in \lie k^*$ such that $\smash{\omega\at\mu = \gamma^* \omega_\Omega \at
  \mu}$, where $\gamma$ is the unique $K$-equivariant map evaluating to
  $\gamma(\mu)$ at $\mu$. Recall that at $\mu$ we have a basis $X_\alpha$,
  $\smash{\alpha \in \hat\Delta}$ for the tangent space
  $T_\mu^{\mathbb C} \orb{\mu}$. Extend $\omega$ by zero on $\lie h$
  and the root spaces $\lie g_\alpha$ with $\alpha \in \Delta'$.
  Since we have
 	\begin{equation} \label{eq:karabegov:fundamentalvfsonadjandcoadjorbits}
	T_x \gamma v_X \at x = T_x \gamma \partt \Ad_{\exp(-tX)} x = 
	\partt \Ad_{\exp(-tX)}^* \gamma(x) = v_X^{\Omega} \at{\gamma(x)},
 	\end{equation}
  we can immediately note that $\gamma^* \omega_\Omega \at{\mu}(X_\alpha, X_\beta)
  = \omega_\Omega (v_{X_\alpha}^\Omega, v_{X_{\beta}}^\Omega
  )\at{\gamma(\mu)} = \langle \gamma(\mu), [X_\alpha, X_{\beta}]
  \rangle $.
  Therefore there is an obvious candidate for $\gamma(\mu)$: choose
  $\gamma(\mu) \in \lie h^*$ (and extend it as zero on all root
  spaces) such that
  \begin{equation}\label{eq:convergencelemma}
  \omega\at\mu(X_\alpha, X_{-\alpha}) 
  =
  \langle \gamma(\mu), [X_\alpha, X_{-\alpha}] \rangle
  \end{equation}
  holds for all
  fundamental roots $\alpha$. For fundamental roots $[X_\alpha, X_{-\alpha}]$ 
  is a basis of $\lie h$, so this definition makes sense. Note also that 
  $\omega$ is real, so $\gamma(\mu)$
  really lies in $\lie k^*$. We prove below that 
  \eqref{eq:convergencelemma} already implies $\smash{\omega\at\mu(X_\alpha, 
  X_{\beta})} = \langle \gamma(\mu), [X_\alpha, X_{\beta}] \rangle = \gamma^* 
  \omega_\Omega\at\mu (X_\alpha, X_{\beta})$ for all 
  roots $\alpha, \beta$. Thereafter we show 
  that $\gamma(\mu)$ is indeed $K^\mu$-invariant.

  We prove $\omega\at\mu (X_\alpha, X_\beta)= \gamma^* \omega_\Omega
  \at \mu(X_\alpha, X_\beta)$ in increasing generality. First, note
  that by definition \eqref{eq:convergencelemma} it holds for a fundamental 
  root $\alpha =-\beta$.
  A $K$-invariant form is also $\lie k^\mu$-invariant at $\mu$,
  implying $\omega\at\mu(\ad _H X_\alpha, X_\beta) +
  \omega\at\mu(X_\alpha, \ad _H X_\beta) = 0$ for all $H \in \lie h$,
  so $\omega\at\mu(X_\alpha, X_\beta) = 0$ whenever $\alpha \neq -
  \beta$.  Consequently $\omega\at\mu(X_\alpha, X_\beta)$ and
  $\gamma^* \omega_\Omega\at\mu(X_\alpha, X_\beta)$ also coincide
  whenever $\alpha \neq -\beta$.  So the only remaining case is
  $\alpha = -\beta$ for non-fundamental roots. This follows from the
  following statement: if $\omega\at\mu(X_\alpha, X_{-\alpha}) =
  \gamma^* \omega_\Omega\at\mu(X_\alpha, X_{-\alpha})$ holds for some
  roots $\alpha = \beta$ and $\alpha = \gamma$, then it also holds for
  $\beta+\gamma$:
  \begin{align*}
    \omega\at\mu(X_{\beta+\gamma}, X_{-\beta-\gamma}) 
    &= N_{\beta,\gamma}^{-1} \omega\at\mu([X_\beta,X_{\gamma}], X_{-\beta-\gamma}) 
    \\
    &= -N_{\beta,\gamma}^{-1} (\omega\at\mu([X_\gamma,X_{-\beta-\gamma}],X_\beta) + 
    \omega\at\mu([X_{-\beta-\gamma},X_\beta],X_\gamma)) \\
    &= -N_{\beta,\gamma}^{-1} (N_{\gamma,-\beta-\gamma} 
    \omega\at\mu(X_{-\beta},X_\beta) 
    + N_{-\beta-\gamma, \beta} \omega\at\mu(X_{-\gamma}, X_\gamma) \\
    &= N_{\beta,\gamma}^{-1} (N_{\gamma,-\beta-\gamma} \langle \gamma(\mu), H_\beta 
    \rangle + N_{-\beta-\gamma,\beta} \langle \gamma(\mu), H_\gamma \rangle \\
    &= \langle \gamma(\mu), N_{\beta, \gamma}^{-1} (N_{\gamma,-\beta-\gamma} 
    H_\beta +N_{-\beta-\gamma,\beta} H_\gamma) \rangle \\
    &= \langle \gamma(\mu), - H_{-\beta-\gamma} \rangle \\
    &= \gamma^* \omega_\Omega\at\mu(X_{\beta+\gamma}, X_{-\beta-\gamma}) \glpunkt
  \end{align*} 
  We used \autoref{lemma:Kinvariantclosed2} in the second step and
  that the Jacobi identity implies $N_{\alpha,\beta}
  H_\gamma+N_{\beta,\gamma} H_\alpha+ N_{\gamma, \alpha} H_\beta = 0$
  in the second to last step.
  
  Now let us prove $K^\mu$-invariance. 
  Since $K^\mu$ is the centralizer of a torus and thus connected, it
  suffices to check $\lie k^\mu$-invariance of $\gamma(\mu)$. Firstly, $\lie
  h$-invariance is clear since $\gamma(\mu) \in \lie h^*$. Secondly, if $\alpha
  \in \Delta'$ then $\langle \ad^*_{X_\alpha} \gamma(\mu), Z \rangle =
  \langle \gamma(\mu), [-X_\alpha, Z] \rangle$, which vanishes
  whenever $Z \in \lie h$ or any of the root spaces $\lie g_\beta$
  with $\beta \neq -\alpha$. Also $\langle \gamma(\mu), [X_\alpha,
  X_{-\alpha}] \rangle = \gamma^* \omega_\Omega \at \mu (X_\alpha,
  X_{-\alpha}) = \omega\at\mu(X_\alpha, X_{-\alpha}) = 0 $.
\end{proof}
\begin{theorem}
  \label{theo:convergence}
  Let $\orb{\mu}$ be an adjoint orbit with a formal differential
  $K$-invariant star product $\star$ of Wick type. Assume that its
  Karabegov form $\omega = \sum_{r \in \mathbb N_0} z^r \omega_r$ is
  the Taylor series around zero of some rational real form $\omega(\hbar)$,
  regular at zero. Then for any $p, q \in
  \Pol(\orb{\mu})$ and $x \in \orb{\mu}$, 
 $(p \star q)(x)$ is the Taylor series of a unique rational
    function $(p *_\hbar q)(x)$ in $\hbar$ having only finitely many
    poles that might depend on $p$ and $q$ but not on $x$. For all $p,
    q \in \Pol(\orb{\mu})$ and $\hbar$ not lying in this set of poles
    $p *_\hbar q$ is again a polynomial.
\end{theorem}
\begin{proof}
  It is known (see e.g. \cite{mueller-bahns.neumaier:2004b}) that the
  Karabegov form of a $K$-invariant star product is $K$-invariant,
  i.e.~all $\omega_r$ are $K$-invariant. Thus $\omega(\hbar) = \sum_j 
  b_j(\hbar)\omega^j$ is also $K$-invariant, so we can assume that the 
  $\omega^j$ are $K$-invariant. Here the $b_j$ are rational functions 
  of $\hbar$, regular at $0$. Since the star product is
  also of Wick type, we get from \autoref{lemma:Kinvariantclosed11}
  that all $\omega^j$ are of the form $\omega_{\gamma^j}$ for some
  maps $\gamma^j : \orb\mu \to \lie k^*$. Thus the corresponding 
  $K$-equivariant family $\smash{\withhbar{f_\bullet}} = \sum_j 
  b_j(\hbar)f_\bullet^j$ 
  is rational and regular at $0$, so can be used to construct Karabegov's
  star product $*_\hbar$. Since for polynomials their product $(f
  *_\hbar g)(x)$ is a rational function without pole at $\hbar = 0$ by
  \autoref{theo:karabegov:starproduct},
  \refitem{theo:karabegov:starproduct:itemv}, its Taylor series
  coincides with $(f \star' g)(x)$, where $\star'$ is the formal star
  product associated to $*_\hbar$. By construction its Karabegov form
  is the Taylor series expansion of $\omega(\hbar)$, thus coincides
  with the Karabegov form of $\star$. Consequently $\star =
  \star'$. All claims now follow easily from Karabegov's construction.
\end{proof}
This result establishes the existence of a subalgebra, namely the polynomials, 
on which 
we obtain associative products for all complex numbers $\hbar$ (that are 
different from the poles), thereby passing from the formal to the strict 
setting. In 
the next subsection, we show how this subalgebra can be enlarged in the case of 
the 2-sphere.

%
%

\section{\texorpdfstring{Continuity of the star product on 
\boldmath$\mathbb S^2$}{Continuity of the star product on S2}}
\label{sec:continuity}

In this section we focus on the particular case of the two dimensional
sphere $\mathbb{S}^2$ interpreted as a coadjoint orbit of $\SU$.  Let
$\slc$ be the Lie algebra of trace-free complex $2 \times 2$ matrices
with standard basis denoted by $(H, X, Y)$.
Consider $\lie h = \langle H \rangle$.  The compact real form $\lie k$
of $\slc$ is $\su$
and the Killing form is given by
\begin{equation}
  B(H,H) = 8,
  \quad
  B(X,Y) = B(Y,X) = 4
  \quad \text{and}\quad
  B(H,X) = B(H,Y) = B(X,X) = B(Y,Y) = 0.
\end{equation}
Note that on the dual of $\slc$ we have the dual basis to $(H, X, Y)$
denoted by $(H^*, X^*, Y^*)$ and another basis obtained by using the
isomorphism of $\slc$ and $\slc^*$ induced by the Killing form,
\begin{equation}
  (\toFunction H, \toFunction X, \toFunction Y) 
  = 
  (8 H^*, 4 Y^*, 4 X^*).
\end{equation}
Here and in the following 
$\toFunction Z = B(Z, \cdot) \in \slc^*$ for $Z \in \slc$.

The complex connected simply-connected Lie group corresponding to
$\slc$ is the \emph{special linear group} $\SLC$ and the closed
connected subgroup with Lie algebra $\su$ is the \emph{special unitary
  group} $\SU$. Elements $k \in \SU$ can be explicitly parametrized by
\begin{equation}
  \label{eq:SU2:FormOfElements}
  k = \begin{pmatrix}
    u & -\overline v \\ v & \overline u
  \end{pmatrix}
  \quad\text{with}\quad u,v \in \mathbb C \quad \text{satisfying}\quad\abs 
  u^2+\abs v^2 = 1 \glpunkt
\end{equation}
Thus, it is not hard to show that adjoint orbits of $\SU$ coincide
with spheres (with respect to the inner product given by the Killing form).

We identify linear functionals $\alpha \in \lie h^*$ which are purely
imaginary on $\lie h \cap \su$ with real numbers 
$\lambda_\alpha := \alpha(H) \in \mathbb R$ by evaluating them at $H$.
Weights correspond to
integers under this identification.  Excluding the trivial adjoint
orbit $\lbrace 0 \rbrace$, all other orbits are spheres with non-zero
radius and intersect $\lie h$ in exactly two points $\mu$ and $- \mu$.
We choose $\mu$ such that $\lambda := \lambda_{ B(-\I \mu, 
\cdot) } = B(-\I \mu, H)$ is
positive. If $\mu = \I H$ then $B(-\I \mu, \cdot) = 
\toFunction H
= 8 H^*$ and $\lambda = 8 $. Denote such an orbit by
$\mathbb{S}^2_\lambda$.  Using this convention the map $\psi_\mu$ from
\autoref{sec:alekseev:comparison} is given by
\begin{equation} 
  \label{eq:orbitprojection}
  \psi_\mu(k) = \Ad_k \mu = \frac{\I \lambda}{8}((u \overline u - v \overline 
  v) H + 2 u\overline v X + 2 \overline u v Y) \glpunkt
\end{equation}
The roots are given by $\alpha_+$, $\alpha_-$ with $\lambda_{\alpha_+}
= 2$ and $\lambda_{\alpha_-} = -2$.  There are exactly two
$\SU$-invariant complex structures on the sphere, depending on which
of these roots we define positive. The following lemma is a consequence of 
\autoref{theo:preliminaries:complexstructure:uniquekaehlerstructure}: 
Note that $H_{\alpha_+} = H$ and so $\lambda(H_{\alpha_+}) 
>0$.
\begin{lemma} 
  The Kähler complex structure $J_{\textnormal{Kähler}}$ on the
  sphere corresponds to an ordering for which $\alpha_-$ is positive.
\end{lemma}
Remember that to obtain a star product of Wick type, we need to choose
$J$ opposite to $J_{\textnormal{Kähler}}$, so in the corresponding ordering 
$\alpha_+$ is positive.

We observe that any $K$-invariant 2-form on the sphere is a scalar multiple
of the KKS form $\omega_{\mathbb{S}^2}$: by $K$-invariance it is
determined by its value $\alpha \at \mu$ and the space of 2-forms on a
two dimensional vector space is one dimensional. In particular, any
rational $K$-invariant $2$-form which is regular at zero is given by
$\omega(\hbar) = \zeta(\hbar) \omega_{\mathbb{S}^2}$ with a rational
function $\zeta$ having no pole at zero.
Without loss of generality we only consider the formal star product $\star$
with $\omega = \omega_{\mathbb{S}^2}$ and the associated product
$*_\hbar$. One checks easily that the
associated $K$-equivariant family is $f_Z := \toFunction Z : \mathbb
S^2_\lambda \to \mathbb C$.  (For better readability we leave out the
restriction of $\toFunction Z$ to $\mathbb S_\lambda^2$.)

We would like to construct a topology on the polynomials such that the
product $*_\hbar$ becomes continuous. For this, we need some explicit estimates
of the coefficients of the star product. Since we have already seen
that the constructions of Karabegov and Alekseev-Lachowska agree we
can work with either of them. 
We will mostly use the construction
of Alekseev-Lachowska since it produces easier formulas, but in the last part 
of this section we describe how to obtain the same results in the Karabegov 
approach.

\subsection[\texorpdfstring{Formulas for $\SU$ and related properties}{Formulas 
for SU(2) and related properties}]{\texorpdfstring{\boldmath Formulas for $\SU$ 
and related properties}{Formulas for SU(2) and related properties}}
In this section we introduce notation and a locally convex 
topology, called the $T_R$-topology, on the polynomials on the sphere. Both are 
needed in the next section to prove the continuity of $*_\hbar$.

As a first step we recall a formula, obtained by Alekseev and Lachowska in
  \cite{alekseev.lachowska:2005a}, for the star product on the two sphere. With 
  our conventions 
  $\alekseevN + = \langle X
  \rangle$ and $\alekseevN - = \langle Y \rangle$ with $X$ of degree 1
  and $Y$ of degree $-1$.  The pairing is diagonal with respect to the
  bases $\lbrace 1, X, X^2, \dots\rbrace$ of $\univ(\alekseevN +)$ and
  $\lbrace 1, Y, Y^2, \dots\rbrace$ of $\univ(\alekseevN -)$ due to
  degree reasons, so we just have to calculate the normalization.
  \begin{lemma} 
    \label{lemma:ExampleSL2:projectionOfXnYn}
    We have $(X^n Y^n)_0 = n! H (H-1) \dots (H-n+1)$.
  \end{lemma}
  \begin{proof} 
    We prove this by induction and a simple calculation:
    \begin{align*}
      (X^n Y^n)_0 
      &=
      \Bigg(\underbrace{Y X^n Y^{n-1}}_{(\dots)_0=0} 
      +
      \sum_{i=0}^{n-1} X^i H X^{n-1-i} Y^{n-1} \Bigg)_0 
      \\
      &=
      \Bigg( \sum_{i=0}^{n-1} \bigg( H X^i  
      + 
      \sum_{j=0}^{i-1} X^j (-2 X) X^{i-j-1} \bigg)X^{n-1-i} Y^{n-1} \Bigg)_0 
      \\
      &=  
      \Bigg( \sum_{i=0}^{n-1} H X^{n-1} Y^{n-1}  
      - 
      2 \sum_{i=0}^{n-1}\sum_{j=0}^{i-1}   X^{n-1} Y^{n-1} \Bigg)_0 
      \\
      &=
      ( n H X^{n-1} Y^{n-1}  - n (n-1)  X^{n-1} Y^{n-1} ) _0 
      \\
      &=  
      n (H-n+1) (X^{n-1} Y^{n-1} ) _0 \glpunkt
    \end{align*}
  \end{proof}
  The previous lemma implies immediately that
  \begin{equation*}
    (Y^n, X^n)_\hbar 
    = 
    \evaluatebig{\frac {B(-\I \mu, \cdot)} \hbar}{(-1)^n (X^n Y^n)_0} 
    =
    (-1)^n n! \frac\lambda\hbar\left(\frac\lambda\hbar-1\right)\cdots\left(\frac\lambda\hbar -(n-1)\right)
  \end{equation*}
  and therefore the pairing $(\cdot,\cdot)_\hbar$ is non-degenerate if $\hbar 
  \notin \poleset := \lbrace 0, \lambda, \frac \lambda 2, \frac \lambda 3, 
  \dots \rbrace$. Consequently, for $\hbar \notin \Omega$ the product $p 
  *_\hbar q = F_\hbar(p,q)$ is well defined for all $p, q \in \Pol(\mathbb 
  S_\lambda^2)$, where
  \begin{align}
    \label{eq:alekseev:BforSU2}
    F_\hbar 
    = 
    \sum_{n=0}^\infty \frac{(-1)^n  
	  \hbar^n }{n!  \lambda(\lambda- \hbar)\cdots 
	  (\lambda - (n-1)\hbar)} Y^n \tensor X^n \glpunkt
  \end{align}
Note that the stabilizer Lie group of $\mu$ is $\SU^\mu = \lbrace \diag(\E^{\I 
t}, 
\E^{-\I t}) \mid t 
\in \mathbb R \rbrace \cong \U$. Therefore $\mathbb S_\lambda^2 \cong \SU / 
\U$ and we can identify smooth 
functions on $\mathbb S_\lambda^2$ with smooth functions on $\SU$, which are 
invariant under the $\U$-action on $\SU$ by right multiplication, i.e.~$f(x) = 
f(xy)$ for $x \in \SU$ and $y \in \U$. Here elements of $\U$ are realized as 
the $2\times 2$-matrices from above. We will refer to the invariance property 
simply by \emph{right-invariance} in the following and write 
$\Cinfty_{\mathbb C}(\SU)^{\U}$ for the space of such functions.

Note that $Y^n \tensor X^n$ 
in \eqref{eq:alekseev:BforSU2} acts on the 
extensions $\hat p$, $\hat q$ of two polynomials $p$, $q$
on the adjoint orbit to right invariant functions on $\SU$ by the left invariant
bidifferential operators corresponding to $Y^n \tensor
X^n$, giving another right invariant function $(\leftinv Y)^n \hat p \cdot 
(\leftinv X)^n \hat q$ and thus also a polynomial on the orbit. However, neither
$(\leftinv Y)^n \hat p$ nor $(\leftinv X)^n \hat q$ is right invariant in 
general, so that these functions do not
necessarily descend to well-defined functions on the adjoint
orbit. It is just their product $(\leftinv Y)^n \hat p \cdot 
(\leftinv X)^n \hat q$ that does.

To obtain continuity estimates it is very convenient to cure this pathology by 
introducing a bigger class of functions than right invariant extensions of 
polynomials. In this class $(\leftinv Y)^n$ should be well-defined and the 
$n$-fold composition of $\leftinv
Y$. Indeed, such a class of functions exists and is given by
$K$-finite functions. Remember that a function $f$ on a manifold $M$ with an 
action of $K$ is $K$-finite, if the span of $k
\acts f$ is finite dimensional. (Here $k \acts f (m) = f(k^{-1} \acts m)$ for 
$k \in K$ and $m \in M$.) We use
the following standard facts on $K$-finite functions.

\begin{proposition}
  \label{theo:kfinite} 
  Let $K$ be a compact connected semisimple Lie group.
  \begin{propositionlist}
  \item The $K$-finite functions on an adjoint orbit $K / K^\mu$
    coincide with polynomials on the adjoint orbit.
    \item If $K$ is realized as a matrix Lie group, then the $K$-finite
    functions on $K$ coincide with polynomials of the matrix entries
    and their complex conjugates.
  \end{propositionlist}
\end{proposition}
\begin{proof}
  It is a standard argument, see e.g.~\cite[Lemma
  15]{karabegov:1998c}.
\end{proof}
\begin{definition}[\boldmath The algebras $\finitesgroup$ and 
$\finiteinvsgroup$]
  The algebra $\finitesgroup$ is defined to be the algebra generated by the functions
  $U,\overline U, V, \overline V \colon \SU \to \mathbb C$, given for $k
 \in \SU$ by
 \begin{equation}
  U\colon k \mapsto u \glkomma \quad \overline U: k\mapsto \overline u \glkomma \quad 
  V\colon k\mapsto v 
  \glkomma \quad \overline V: k\mapsto \overline v\glpunkt
  \label{eq:SU2:DefUV}
\end{equation}
  The algebra $\finiteinvsgroup$ is defined to be the algebra generated by the functions
  $A, B, C \colon \SU \to \mathbb C$, given for $k
 \in \SU$ by
 \begin{equation}
  \label{eq:SU2:DefABC}
  A = U \overline U - V \overline V \colon k \mapsto \abs u^2-\abs v^2 \glkomma 
  \quad B = \overline U V \colon k \mapsto \overline u 
  v\quad\text{and}\quad C = U \overline V \colon k \mapsto u \overline v  
  \glpunkt
\end{equation}
\end{definition}
From \autoref{theo:kfinite} it is clear that $\finitesgroup$ are just the 
$\SU$-finite functions on $\SU$. We will see below that $\finiteinvsgroup$ are 
the right invariant elements of $\finitesgroup$.

Note that $U \overline U + V \overline V =1$. We would also like to consider a 
\anfa free\anfel commutative algebra without this relation. To this end define
$\hat U, \hat{\overline U}, \hat V, \hat{\overline V}, \hat A, \hat B, \hat C, 
\hat D : \mathbb C^2 \to \mathbb C$
by
\begin{equation}
\hat U : (z_1, z_2) \mapsto z_1 \glkomma \quad
\hat{\overline U} : (z_1, z_2) \mapsto \overline z_1 \glkomma \quad
\hat V : (z_1, z_2) \mapsto z_2 \quad \text{and} \quad
\hat{\overline V}: (z_1, z_2) \mapsto \overline z_2 \glpunkt
\end{equation}
$\hat A$, $\hat B$ and $\hat C$ are the same products as before, just 
decorating \eqref{eq:SU2:DefABC} with hats and $\hat D = \hat U \hat{\overline 
U} + \hat V \hat{\overline V}$. Note that the polynomial algebra 
$\finitesspace$ on $\mathbb C^2$ is generated by $\hat U$, $\hat{\overline U}$, 
$\hat V$ and $\hat{\overline V}$ and denote the algebra generated by $\hat A$, 
$\hat B$, $\hat C$ and $\hat D$ by $\finiteinvsspace$.

Recall that $\SU^\mu = \lbrace \diag(\E^{\I t}, \E^{-\I t}) \mid t 
\in \mathbb R \rbrace$. For a
monomial $U^\alpha \smash{\overline U}^\beta V^\gamma \smash{\overline
V}^\delta$ we introduce the numbers $\numberofnonbars=\alpha+\gamma$ 
and $\numberofbars = \beta+\delta$. Then a monomial in $U$, $\overline U$,
$V$ and $\overline V$ is right invariant if and only if  
$\numberofnonbars = \numberofbars$.
Consider the action of $\U$ on $\mathbb C^{2}$ by componentwise
multiplication and the induced action on $\finitesspace$.
Defining $\numberofnonbarsext$ and $\numberofbarsext$ in the same way as
$\numberofnonbars$ and $\numberofbars$, the right invariant monomials are 
exactly the ones with
$\numberofnonbarsext=\numberofbarsext$. So this says that the right invariant 
elements of $\finitesspace$ are exactly
$\finiteinvsspace$ and the ones of $\finitesgroup$
are $\finiteinvsgroup$.
\begin{proposition}\label{theo:polysareisomtoregulars}
  $\finiteinvsgroup$ is isomorphic to $\Pol(\mathbb S_\lambda^2)$.
\end{proposition}
\begin{proof}
  We only need to see that the restriction and extension maps between
  right invariant functions on $\SU$, i.e.~$\Cinfty_{\mathbb C}(\SU)^{\U}$, and
  functions on the orbit, i.e.~$\Cinfty_{\mathbb C}(\SU/\U)$, map 
  $\finiteinvsgroup$ to
  $\Pol(\mathbb S_\lambda^2)$ and vice versa. However, this is immediately
  clear from \autoref{theo:kfinite} and since by $\SU$-equi\-va\-riance of the 
  map 
  $\SU \to \SU /\U$ they map 
  $\SU$-finite functions to $\SU$-finite functions.
\end{proof}
We can determine the isomorphism explicitly:
from \eqref{eq:orbitprojection} it follows that
\begin{equation}
  \label{eq:HXYasABC}
  \toFunction H \circ \psi_\lambda = \I \lambda A \glkomma \quad
  \toFunction X \circ \psi_\lambda = \I \lambda B \quad \text{and}\quad
  \toFunction Y \circ \psi_\lambda = \I \lambda C \glpunkt
\end{equation}
It is well-known that any polynomial vanishing on the 3-sphere viewed as a 
submanifold of $\mathbb R^4$ is a multiple of
$x_1^2+x_2^2+x_3^2+x_4^2-1$, where $(x_1, x_2, x_3, x_4)$ are standard 
coordinates on $\mathbb R^4$. This implies that any element
of $\finitesspace$ vanishing on $\SU \subseteq
\mathbb C^{2\times 2}$ (from $p \in \finitesspace$ we obtain a function on 
$\SU$ by composing with the map projecting a $2 \times 2$ matrix to its first 
column) 
if and only if it is a multiple of $\hat U \hat{\overline U} + \hat V
\hat{\overline V}-1$.  Similarly an element of $\mathbb C[\hat A, \hat B, \hat 
C]$, the algebra generated by $\hat A$, $\hat B$ and $\hat C$, vanishes on 
$\SU$ if and only if it is a multiple of $\hat A^2+4 \hat B\hat C-1$. 
Consequently, we can prove
the following claim.
\begin{proposition} \label{theo:discriptionPR}
   We have 
   \begin{equation}
      \finitesgroup \cong \finitesspace / \langle \hat U \hat{\overline U} + \hat V 
\hat{\overline V} -1 \rangle
\quad\text{and}\quad
\finiteinvsgroup \cong 
\finiteinvsspace / \langle \hat U \hat{\overline U} + \hat V \hat{\overline V} 
-1 \rangle \cong \mathbb C[\hat A, \hat B, \hat C] / \langle \hat A^2 + 4 
\hat B \hat C 
-1 \rangle \glpunkt
   \end{equation}
\end{proposition}
Now since $\leftinv{X}$ and $\leftinv{Y}$ are left invariant vector
fields, it is clear that by applying them to a $\SU$-finite function, we
get a $\SU$-finite function again. 
\begin{lemma}
  We have:
\begin{subequations}
  \begin{align}
    \label{eq:leftinvX}
    \leftinv X U &= 0\glkomma &
    \leftinv X V &= 0\glkomma &
    \leftinv X \overline U &= V\glkomma &
    \leftinv X \overline V &= -U\glkomma \\
    \leftinv Y U &= -\overline V\glkomma &
    \leftinv Y V &= \overline U\glkomma &
    \leftinv Y \overline U &= 0 \glkomma &
    \leftinv Y \overline V &= 0 \glpunkt
    \label{eq:leftinvY}
  \end{align}
\end{subequations}
\end{lemma}
\begin{proof}
  Let us compute:
\begin{align*}
  \leftinv{X} \overline V(k) 
  &= 
  \partt \exp\left(t \begin{pmatrix}
				 0 & 1 \\ 0 & 0 
  \end{pmatrix}\right)\acts \overline V(k) 
  \\
  &= 
  \partt \overline V\left( \begin{pmatrix}
				 u & v \\ -\overline v & \overline u 
  \end{pmatrix} \begin{pmatrix}
				 1 & t \\ 0 & 1 
  \end{pmatrix} \right) 
  \\ 
  &= 
  \partt \overline V \begin{pmatrix}
				 u & u t + v \\ - \overline v & -t \overline v + \overline u
  \end{pmatrix} 
  \\
  &= 
  - \partt (u t+v) 
  \\
  &= 
  - u
  \\
  &= 
  - U(k) \glpunkt 
\end{align*}
 The other equalities are obtained similarly.
\end{proof}
It will be convenient in the following to not just have a product on
$\finiteinvsgroup$, but also on $\finitesspace$ and $\finiteinvsspace$. Thus
define $\prodP: \finitesspace \times
\finitesspace \to \finitesspace$ by applying the element $F_\hbar$ from
\eqref{eq:alekseev:BforSU2}, with the derivatives defined by
decorating \eqref{eq:leftinvX} and \eqref{eq:leftinvY} with hats. It
is clear from the construction that $\prodP$ restricts to a
product $\prodR$ on $\finiteinvsspace$ and descends to
$\finiteinvsgroup$, where it coincides with $*_\hbar$. Keep in mind that
$\prodP$ is not associative any more, but that the associativity
only holds on the quotient.
\begin{lemma}
  \label{lemma:alekseevderivatives}
  Let $p = p_1 \cdots p_d$ be a monomial with
  $p_i \in \lbrace A, B, C \rbrace$. Let $Z$ be any of $U$,
  $\overline U$, $V$ or $\overline V$.
Then
  \begin{lemmalist}
  \item\label{lemma:alekseevderivatives:i} $(\leftinv X)^2 Z =
    (\leftinv Y)^2 Z = 0$.
  \item\label{lemma:alekseevderivatives:ii} $(\leftinv X)^n p = (\leftinv 
    Y)^n p = 0$ if $n > d$.
  \end{lemmalist}
\end{lemma}
\begin{proof}
  Part \refitem{lemma:alekseevderivatives:i} follows immediately from
  \eqref{eq:leftinvX} and \eqref{eq:leftinvY}. Then using the product rule it 
  follows that $(\leftinv X)^2 p_i = 0$ since at least one of $\leftinv X Z$ or 
  $\leftinv X \overline Z$ vanishes and similarly for $Y$. This implies part
  \refitem{lemma:alekseevderivatives:ii} by using the product rule again.
\end{proof}
Let $V$ be a finite dimensional locally convex vector space and $R \in \mathbb 
R$. For
any continuous seminorm $p$ define a seminorm on the tensor algebra
$T^\bullet(V)$ by
\begin{equation}
  p_R = \sum_{n=0}^\infty (n!)^R p^n \glkomma
\end{equation}
where $p^n = p \tensor \ldots \tensor p$. We write $T^\bullet_R(V)$ for the
locally convex space that arises if $T^ \bullet(V)$ is endowed with
the seminorms $p_R$ for all continuous seminorms $p$.
We get an induced topology on the subspace
$\Sym^\bullet(V)$ that is closed if $R \geq 0$. In this case the
evaluation functional $\delta_\alpha : \Sym^\bullet_R(V) \to \mathbb
C$ at $\alpha \in V^*$ is continuous.
\begin{remark}
		The $T_R$-topology is coarser than the $T_{R'}$-topology for $R < R'$ 
		and has a larger completion, see e.g.~\cite[Lemma 3.6]{waldmann:2014a}. 
		Thus we are mainly interested in the case $R=0$ below, even though we 
		formulate our results for all $R \geq 0$.  The statements above remain 
		true in 
		infinite dimensions if one uses the projective tensor product.
\end{remark}
For a finite dimensional vector space all norms are equivalent and
thus we can fix any norm $\norm\cdot$ and it suffices to 
consider the seminorms $(C \norm\cdot)_R$ for all $C \geq 1$.
In the following we will have to estimate these seminorms. The
tensor product of $1$-norms has a particularly easy form.
\begin{lemma}
  \label{theo:continuity:onenormprojectivetensorproduct}
  Let $V$ be a finite dimensional vector space with basis $e_1, \ldots,
  e_n$. Let $\norm\cdot_1$ denote the usual $1$-norm with respect to
  this basis, i.e.~if $v \in V$, $v = \sum_i v_i e_i$ then $\norm v_1
  = \sum_i \abs{v_i}$. Then on the tensor product
  $V^{\tensor k}$ we have
  \begin{equation}
    \label{eq:TensorProduct:projective:forOneNorm}
    \norm\cdot_1 \tensor \dots \tensor \norm\cdot_1(t)
    = 
    \sum_{i_1, \ldots, i_k=1}^n \abs{t_{i_1 \cdots i_k}} 
    \quad\text{if}\quad 
    t 
    =  
    \sum_{i_1, \ldots, i_k=1}^n t_{i_1 \dots i_k} e_{i_1} \tensor \cdots \tensor 
    e_{i_k}\glpunkt
  \end{equation}
\end{lemma}
We write the symmetric tensor product as a simple product and 
use the convention that $x_1 \cdots x_n = \frac 1 {n!}
\sum_{\sigma \in S_n} x_{\sigma(1)} \tensor \dots \tensor
x_{\sigma(n)}$, so we have
\begin{equation}
  (C\norm\cdot_1)_R \bigg(\sum_{I \in 
  \mathbb N_0^n} a_{I}e_{1}^{I_1} 
  \dots 
  e_{n}^{I_n} 
  \bigg) = \sum_{I \in 
  	\mathbb N_0^n} (\abs I !)^R C^{\abs I} \abs{a_{I}} 
  \glpunkt
\end{equation}
Here we use multiindices $I = (I_1, \dots, I_n) \in \mathbb N_0^n$ and $\abs I 
= I_1 + \dots + I_n$.
\begin{corollary}
	\label{theo:continuity:estimateformonomials}
	Let $V$ be a finite dimensional vector space with basis $e_1, \dots, e_n$ 
	and $\norm\cdot_1$ the 1-norm with respect to this basis. Then to prove 
	continuity of a linear map $f : S_R^\bullet(V) \to W$ where $W$ is some 
	locally	convex vector space, it suffices to find for each continuous 
	seminorm $p$ on $W$ a constant $C_p >0$ such that
	\begin{equation} \label{eq:continuityestimateformonomials}
	p(f(e_{1}^{I_1}\dots e_n^{I_n})) \leq 
	(C_p\norm\cdot_1)_R(e_1^{I_1} 
	\dots e_n^{I_n}) 
	\end{equation}
	holds for any multiindex $I = (I_1, \dots, I_n) \in \mathbb N_0^n$.
\end{corollary}
\begin{proof}
	It follows from \eqref{eq:continuityestimateformonomials} and 
	\autoref{theo:continuity:onenormprojectivetensorproduct} that
	\begin{align*}
	p\bigg(f \bigg(\sum_{I \in \mathbb N_0^n} 
	a_{I}e_1^{I_1} \dots 
	e_n^{I_n} 
	\bigg) \bigg)
	&\leq \sum_{I \in \mathbb N_0^n}  
		\abs{a_I} p (f(e_1^{I_1} \dots 
		e_n^{I_n} )) \\
	&\leq \sum_{I \in \mathbb N_0^n}  
	\abs{a_{I}} (C_p \norm \cdot_1)_R (e_1^{I_1} 
	\dots 
	e_n^{I_n} ) \\
	&= (C_p \norm \cdot_1)_R \bigg( \sum_{I \in \mathbb 
	N_0^n}  
	a_{I} e_1^{I_1} \dots 
	e_n^{I_n} \bigg)
	\glkomma
	\end{align*}
	where we assume that only finitely many coefficients $a_I$ are non-zero, so 
	that all sums are in fact finite.
\end{proof}
We need the following well-known estimate.
\begin{lemma} 
  \label{theo:continuity:decreasingfactorialgrowsasfactorial}
  For a fixed compact disc $D_r(z_0) = \lbrace z \in \mathbb C \mid
  \abs{z-z_0} \leq r \rbrace$ with $D_r(z_0) \cap \mathbb N_0 =
  \emptyset$ there are constants $0 < C_-\leq 1 \leq C_+$ depending on
  $r$ and $z_0$ such that for all $\delta \in \mathbb N$ and $z \in
  D_r(z_0)$ we have
  \begin{equation}
    \delta! \, C_-^\delta \leq \abs{z (z-1) \cdots (z-\delta +1) } \leq \delta! \, C_+^\delta \glpunkt
  \end{equation}
\end{lemma}
There are several ways to obtain continuity estimates for $*_\hbar$
and it is not quite clear which one is best suited for
generalization. Therefore we define two topologies below, that turn out to be 
equivalent but can both be used to obtain continuity estimates.
\begin{definition}[\boldmath Reduction and quotient-$T_R$-topologies]
\begin{definitionlist}
 \item   Let $\mynorm$ be the $1$-norm with respect to the basis $\lbrace \hat U,
  \hat V, \hat{\overline U}, \hat{\overline V} \rbrace$ of $(\mathbb C^{2})^*$ 
  (thinking of $\mathbb C^2$ as a real vector space). The corresponding 
  $T_R$-topology on $\Sym^\bullet((\mathbb C^{2})^*) \cong 
  \finitesspace$ induces a subspace topology on $\finiteinvsspace$ and thus a 
  quotient topology on $\finiteinvsgroup$, called the \emph{\reductionTR}.
  \item Similarly, the $1$-norm $\norm \cdot_1$ with respect to the basis $\lbrace
  \toFunction H, \toFunction X, \toFunction Y\rbrace$ on $\slc^*$ induces a 
  $T_R$-topology on $\Sym^\bullet(\slc^*)\cong \Pol(\slc)$ and we call its 
  quotient topology on $\Pol(\slc)/\vanideal$ the \emph{\quotientTR}.
\end{definitionlist}
\end{definition}
Here $\vanideal$ is the ideal of polynomials
vanishing on the orbit, which is generated by
$\toFunctionPower{H}{2}+\toFunction X \toFunction Y + \lambda^2$.

\begin{proposition} \label{theo:quotandredtopologiesagree}
	The \reductionTR\
	and \quotientTR[2R]\ agree.
\end{proposition}
\begin{proof}
	Recall from \autoref{theo:polysareisomtoregulars} that the 
	space of 
	polynomials on the sphere is isomorphic to $\finiteinvsgroup$ with the 
	isomorphism given as in \eqref{eq:HXYasABC}. Let $W$ be the three 
	dimensional 
	vector space spanned by $\hat A$, $\hat B$ and $\hat C$, endowed with the 
	1-norm $\norm \cdot'_1$ with respect to this basis.
	Consider the tensor algebra 
	$\Sym^\bullet(W)$ with the induced $T_R$-topology, then it is clear from 
	\autoref{theo:discriptionPR} that the induced quotient topology on 
	$\Sym_R^\bullet(W) / \langle \hat A^2 + 4\hat B \hat C 
	-1\rangle$ coincides with the {\quotientTR} under the above isomorphism.

	So we would like to see that $\Sym_R^\bullet(W) / \langle \hat A^2 + 4 \hat 
	B \hat C -1 
	\rangle$ is homeomorphic to $\finiteinvsspace / \langle \hat U 
	\hat{\overline U} + \hat V \hat{\overline V} -1 \rangle$.
	To this end we define two maps $f: \Sym^\bullet(W) \to \finiteinvsspace$ 
	and $g:\finiteinvsspace \to \Sym^\bullet(W)$ by extending 
	\begin{align*} f(\hat A^\alpha \hat B^\beta \hat C^\gamma) &= (\hat U 
	\hat{\overline U} - \hat V \hat{\overline 
	V})^\alpha 
	(\hat {\overline U}\hat V)^\beta (\hat U \hat{\overline V})^\gamma \\
	g(\hat U^\alpha \hat{\overline 
	U}{}^\beta \hat V^\gamma \hat{\overline V}{}^\delta) &=  
	\left(\frac{1+\hat A}2\right)^{\alpha-\alpha\wedge\delta} 
	\left(\frac{1-\hat A}2\right)^{\delta-\alpha\wedge\delta}
	\hat B^{\beta \wedge \gamma} 
	\hat C^{\alpha 
		\wedge \delta}
	\end{align*} linearly. Here we used the shorthand $m \wedge n = \min\lbrace 
	n, m\rbrace$. It is easy to check that these maps descend to bijections on 
	the quotients, so we shall be done if we can show that they are continuous. 
	Set $d = \alpha + \beta + \gamma$, then
\begin{align*}
(C\mynorm)_{R} (f(\hat A^\alpha \hat B^\beta \hat C^\gamma)) 
&= (C \mynorm)_R\left(
\sum_{i=0}^\alpha \binom \alpha i (\hat U \hat{\overline U})^i (-\hat V  
\hat{\overline V})^{\alpha-i} 
(\hat{\overline U}\hat V)^\beta (\hat U \hat{\overline V})^\gamma 
\right) \\
&\leq \sum_{i=0}^\alpha \binom \alpha i ((2d)!)^R C^{2 d} \\
&= 2^\alpha \binom{2 d}{d}^R (d!)^{2 R} C^{2 d} \\
&\leq 2^{d} 2^{2 d R} (d!)^{2R} C^{2 d}\\
&= (2^{2 R+1} C^2 \norm\cdot'_1)_{2 R} (\hat A^\alpha \hat B^\beta \hat 
C^\gamma) 
\end{align*}
and the continuity of $f$ follows from 
\autoref{theo:continuity:estimateformonomials}. Similarly, setting $d' = \alpha 
+ \beta + \gamma +\delta$ and using that 
$\alpha + \gamma = \beta + \delta$ for elements of $\finiteinvsspace$ we obtain
\begin{align*}
(C \norm\cdot'_1)_{2 R} (g(\hat U^\alpha \hat{\overline 
U}{}^\beta \hat V^\gamma \hat{\overline V}{}^\delta)) &= 
(C \norm\cdot'_1)_{2 R} \left(
\left(\frac{1+\hat A}2\right)^{\alpha-\alpha\wedge\delta} 
\left(\frac{1-\hat A}2\right)^{\delta-\alpha\wedge\delta}
\hat B^{\beta \wedge \gamma}\hat C^{\alpha 
	\wedge \delta} \right) \\
&\leq (C \norm\cdot'_1)_{2 R} \left( \hat A^{\abs{\alpha-\delta}}
\hat B^{\beta \wedge \gamma} \hat C^{\alpha 
	\wedge \delta} 
\right) \\
&= ((d'/2)!)^{2 R} C^{d'/2} \\
&\leq (d'!)^R C^{d'/2} \\
&= (C^{1/2} \mynorm)_R(\hat U^\alpha \hat{\overline 
U}{}^\beta \hat V^\gamma \hat{\overline V}{}^\delta) \glpunkt
\end{align*}
Here we used that $\alpha + \delta - 2\alpha \wedge \delta = 
\abs{\alpha-\delta}$ and that 
\begin{align*} 
  2 (\beta \wedge \gamma + \alpha \wedge \delta + \abs{\alpha-\delta}) 
  &= 
  2 \beta \wedge \gamma + \abs{\beta-\gamma} + 2 \alpha 
   \wedge \delta + \abs{\alpha-\delta}
   \\
&= \beta + \gamma + \alpha + \delta 
\\
&= d' \glpunkt 
\end{align*}

\end{proof}
We want to finish this subsection with the following easy observation.
\begin{proposition}
  The \quotientTR\ and the \reductionTR\ are Hausdorff.
\end{proposition}
\begin{proof}
  It suffices to prove that the ideal divided out is closed. For the 
  \quotientTR\ it is the
  intersection of the kernels of all evaluation functionals at points
  of the $2$-sphere and these functionals are continuous with respect
  to any $T_R$-topology.
\end{proof}

\subsection[\texorpdfstring{Continuity in the reduction and 
quotient-$T_R$-topologies}{Continuity in the reduction and 
quotient-TR-topologies}]{\texorpdfstring{\boldmath Continuity in the reduction 
and quotient-$T_R$-topologies}{Continuity in the reduction and 
	quotient-TR-topologies}}

In this section we will prove the continuity of $*_\hbar$ with 
respect to the \reductionTR, using the approach of Alekseev-Lachowska. We will 
outline how the approach of Karabegov can be used to prove continuity with 
respect to the \quotientTR\ and prove that the star product depends 
holomorphically on $\hbar$ on the completion.
\begin{theorem}\label{theo:continuity}
  The star product $*_\hbar$ on $\mathbb S^2_\lambda$ is continuous
  with respect to the \reductionTR\ for $R \geq 0$ if $\frac \lambda
  \hbar \notin \mathbb N$.
\end{theorem}
\begin{proof}
  Recall from \autoref{lemma:alekseevderivatives} that both $(\leftinv X)^n$ 
  and $(\leftinv Y)^n$ vanish on
  monomials of degree $\leq n-1$ and neither of them raises the degree
  of a monomial. Set $p_{\alpha\beta\gamma\delta} = {\hat U}^\alpha 
  {\hat {\overline U}}{}^\beta {\hat V}^\gamma {\hat{\overline V}}{}^\delta$ 
  for $\alpha, \beta, \gamma, \delta \in \mathbb N_0$ and let 
  $d=\alpha+\beta+\gamma+\delta$. 
  Note 
  that $(\leftinv Y)^n p_{\alpha\beta\gamma\delta} (\leftinv X)^n 
  p_{\alpha'\beta'\gamma'\delta'}$ is a sum of at most 
  $\frac{d!}{(d-n)!} \frac{d'!}{(d'-n)!}$ such monomials of degree $d+d'$.
  So for 
  $C \geq 1$ we obtain that
  \begin{align*}
    (C \mynorm)_R(p_{\alpha\beta\gamma\delta} \prodP 
    p_{\alpha'\beta'\gamma'\delta'}) &= (C
    \mynorm)_R\left(\sum_{n=0}^\infty \frac{(-1)^n}{n!  \frac \lambda
        \hbar \left(\frac \lambda \hbar-1\right)\dots\left(\frac
          \lambda
          \hbar - (n-1)\right)} (\leftinv Y)^n p_{\alpha\beta\gamma\delta} 
          (\leftinv X)^n p_{\alpha'\beta'\gamma'\delta'}\right) \\
    &\leq \sum_{n=0}^{\min\lbrace d,d'\rbrace} \frac{1}{(n!)^2 C_-^n }
    \frac{d!}{(d-n)!} \frac{d'!}{(d'-n)!} (d+d')!^R C^{d+d'}\\
    &\leq 2^{\min\lbrace d,d'\rbrace} C_-^{-\min\lbrace d,d'\rbrace} 2^d
    2^{d'} (2^{d+d'} d! d'!)^R C^{d+d'}\\
    &\leq (2^{2+R} CC_-)^{d} (d!)^R (2^{2+R} CC_-)^{d'} (d'!)^R\\
    &= {(2^{2+R} CC_- \mynorm)}_{R}(p_{\alpha\beta\gamma\delta}) {(2^{2+R} CC_- 
    \mynorm)}_R(p_{\alpha'\beta'\gamma'\delta'})
    \glpunkt
  \end{align*}
  By using \autoref{theo:continuity:estimateformonomials} (or a similar version 
  for maps with two arguments) the continuity of $\prodP$ with respect to 
  the $T_R$-topology induced by $\mynorm$ on $\finitesspace$ follows and thus 
  the continuity of $*_\hbar$ with respect to the \reductionTR.
\end{proof}
The following theorem is obvious because of 
\autoref{theo:quotandredtopologiesagree}. However, it can also be proved 
independently by deriving a formula for $*_\hbar$ that only uses 
$\finiteinvsgroup$. (Note that $p_{W,\overline W}:=\leftinv Y W \cdot \leftinv 
X \overline W$ is an element of $\finiteinvsgroup$ for $W \in \lbrace U, V 
\rbrace$, $\overline W \in \lbrace \overline U, \overline V\rbrace$, so the 
$p_{W,\overline W}$ can be used to replace the left invariant vector fields 
applied to polynomials.)
\begin{theorem}\label{theo:continuity:ii}
  The star product $*_\hbar$ on $\mathbb S^2_\lambda$ is continuous
  with respect to the \quotientTR\ for $R \geq 0$ if $\frac \lambda
  \hbar \notin \mathbb N$.
\end{theorem}
According to \autoref{theo:continuity} the star product $*_\hbar$ is
continuous on $\Pol(\mathbb{S}^2_\lambda)$ endowed with the
{\reductionTR} for $R \geq 0$. Extend $*_\hbar$ to the completion
$\Polcompleteds{R}{\mathbb{S}^2_\lambda}$, still denoted by $*_\hbar$.
\begin{lemma}
  For all $x \in \mathbb S_\lambda^2$ the evaluation functional
  $\ev^{\mathbb S}_x : \Pol(\mathbb S_\lambda^2) \to \mathbb C$,
  $p \mapsto p(x)$ is continuous in the \reductionTR .
\end{lemma}
\begin{proof}
  Take any $k =\left(\begin{smallmatrix}
  u & - \overline v \\ v & \overline u
  \end{smallmatrix}\right) \in K$ with $\psi_\mu(k) =x$. Then $\ev^{\mathbb 
  S}_x$ is 
  the induced quotient map of the continuous map
  $\ev_{(u,v)}$ and therefore continuous.
\end{proof}
Recall that $\poleset = \lbrace 0, \lambda, \frac \lambda 2 , \frac \lambda 
3,
\frac \lambda 4, \dots \rbrace$.
\begin{theorem} \label{theo:continuity:karabegovI:dependenceOnHbar}
  For fixed $p, q \in \Polcompleteds
  R {\mathbb{S}^2_\lambda}$ and $x \in \mathbb{S}^2_\lambda$ the
  function
  \begin{align*}
    \mathbb C \setminus \poleset &\longrightarrow \mathbb C \\
    \hbar &\longmapsto   p*_\hbar q(x)
  \end{align*}
  is holomorphic in $\hbar$.\fxnote{Check whether all indices
    etc.  should be $\mu$ or $\lambda$}
\end{theorem}
\begin{proof}
  Choose sequences of polynomials $p_n, q_n \in
  \Pol(\mathbb{S}^2_\lambda)$ with $p_n \to p$ and $q_n \to q$ for $n
  \to \infty$. The star product $*_\hbar$ is continuous, so we also
  have $p_n *_\hbar q_n \to p*_\hbar q$. Continuity of
  $\ev_x^{\mathbb S}$ implies that $p_n *_\hbar q_n(x) \to p*_\hbar
  q(x)$. Since by \autoref{theo:karabegov:starproduct} all maps $\hbar
  \mapsto p_n *_\hbar q_n(x)$ are rational with finitely many poles in
  $\poleset$, they are in particular holomorphic on $\mathbb C \setminus
  \poleset$. So it suffices to prove
  that $p_n *_\hbar q_n(x) \to p*_\hbar q(x)$ locally uniformly in
  $\hbar$.  But
  \autoref{theo:continuity:decreasingfactorialgrowsasfactorial} gives
  a locally uniform estimate, so all estimates in the previous section
  are locally uniformly in $\hbar$, so by applying $\ev_x^{\mathbb S}$
  the result follows.
\end{proof}

\begin{remark}[Continuity by brute force]
  We remark that \autoref{theo:continuity:ii} can be also proved
  in a more direct but complicated way. Here we only sketch the
  main steps and refer to \cite[Chapter 4]{philipp:thesis} for the details.
  First, observe that one can calculate the holomorphic derivatives $\xi_W \toFunction Z$
  for $W, Z \in \lbrace H, X, Y\rbrace$ explicitly. This calculation
  shows that one can extend them uniquely to polynomials on $\lie{su}_2$
  in such a way that $\xi_W \toFunction Z = \frac \I \lambda \toFunction
  W \toFunction Z + \text{lower order terms}$. Using this extension, one
  can define 
  \begin{equation}
  \Oplhlift: \Sym^\bullet(\slc) \to \Pol(\su)
  \end{equation}
  in the same way as 
  $\Oplh$ before, but using
  the derivatives $\xi_W \toFunction Z$ from above and
  interpreting $f_Z=\toFunction Z$ as an element of $\Pol(\su)$. It
  turns out that $\Oplhslift$ is invertible (away from the
  poles).
  \begin{lemma}\label{lemma:karabegovdeformsgutt}
    Karabegov's star product on $\Pol(\mathbb S_2^\lambda)$ can be
    written as
    \begin{equation}
      \label{eq:KarabegovGutt}
      f\at{\mathbb{S}^2_\lambda} *_\hbar g\at{\mathbb{S}^2_\lambda} 
      = 
      \Oplhlift
      \left(
      \left(\Oplhlift\right)^{-1}(f) *_{\mathrm{Gutt},1} \left(\Oplhlift\right)^{-1}(g)
      \right) 
      \At{\mathbb{S}^2_\lambda}
    \end{equation}
    for $f, g \in \Pol(\su)$. Here we denoted by $*_{\mathrm{Gutt},1}$
    the Gutt star product for $\hbar =1$ viewed as an associative
    deformation of the symmetric algebra $\Sym^\bullet(\slc)$.
  \end{lemma}
  To obtain continuity estimates we need explicit formulas for  $\Oplhslift$ and its inverse, see 
  \cite[Lemma 4.2.16]{philipp:thesis}. These formulas and the observation that
we get $\defspace =
\Pol(\mathbb{S}_\lambda^2)$ whenever $\lambda / \hbar \notin
\mathbb N_0$ allow us to prove 
the following proposition.
\begin{proposition}\label{prop:continuity:oplhishomeomorphism}
  For $R \geq 0$ and $\lambda/\hbar \notin \mathbb N_0$ the map
  $\Oplhslift$ is a homeomorphism from $\Sym^\bullet_{R+1}(\slc)$ to
  $\Sym^\bullet_R(\slc^*)$.
\end{proposition}
\begin{proof}
 See \cite[Theorem 4.2.20]{philipp:thesis}. 
\end{proof}
  The continuity of Karabegov's star product $\smash{*_\hbar}$ on $\mathbb S^2_\lambda$ 
  with respect to the \quotientTRs (if $\lambda / \hbar \notin \mathbb N_0$) 
  follows from 
  \eqref{eq:KarabegovGutt} and
  \autoref{prop:continuity:oplhishomeomorphism} since the Gutt star
  product is continuous with respect to the $T_R$-topology on
  $\Sym^\bullet(\slc)$ for $R \geq 1$ according to
  \cite{esposito.stapor.waldmann:2017a}.
\end{remark}

{
  \footnotesize
  \renewcommand{\arraystretch}{0.5}

}

\end{document}